\documentclass[12pt,reqno]{amsart}
\usepackage{amsmath}
\usepackage{amssymb}
\usepackage{amstext}
\usepackage{mathrsfs}
\usepackage{a4wide}
\usepackage{graphicx}
\usepackage{bm}
\allowdisplaybreaks \numberwithin{equation}{section}
\usepackage{color}
\usepackage{cases}

\usepackage{hyperref}
\hypersetup{hypertex=true,
	colorlinks=true,
	linkcolor=blue,
	anchorcolor=blue,
	citecolor=blue }

\numberwithin{equation}{section}

\newtheorem{theorem}{Theorem}[section]
\newtheorem{proposition}[theorem]{Proposition}
\newtheorem{corollary}[theorem]{Corollary}
\newtheorem{lemma}[theorem]{Lemma}

\theoremstyle{definition}

\theoremstyle{remark}
\newtheorem{remark}[theorem]{Remark}

\newcommand{\R}{\mathbb R}
\newcommand{\T}{\mathbb T}
\newcommand{\cH}{\mathcal H}
\newcommand{\cC}{\mathcal C}

\newcommand{\dd}{\,\mathrm d}
\newcommand{\1}{\mathbf 1}

\newcommand{\Div}{\operatorname{div}}

\newcommand{\loc}{\mathrm{loc}}
\newcommand{\M}{{M_\Phi}}

\begin{document}

	\title[Radial symmetry of stationary and uniformly-rotating solutions]{Rigidity of stationary and uniformly rotating planar Euler flows in the $C^1$ Yudovich class}

	\author{Boquan Fan, Yuchen Wang, Chunjing Xie, Weicheng Zhan}
	\address{Boquan Fan, School of Mathematical Sciences and Institute of Natural Sciences, Shanghai Jiao Tong
		University, 800 Dongchuan Road, Shanghai, P. R. China}
	\email{fanboquan22@mails.ucas.ac.cn}
	
	\address{Yuchen Wang, School of Mathematics Science, Tianjin Normal University, Tianjin,  300387, P. R. China}
	\email{wangyuchen@mail.nankai.edu.cn}
	
	\address{Chunjing Xie, School of Mathematical Sciences,  Ministry of Education Key Laboratory of Scientific and Engineering Computing, CMA-Shanghai, Shanghai Jiao Tong University, Shanghai 200240, China.}
	\email{cjxie@sjtu.edu.cn}
	
	\address{Weicheng Zhan, School of Mathematical Sciences, Xiamen University, Xiamen, Fujian, 361005, P. R. China}
	\email{zhanweicheng@amss.ac.cn}

	\begin{abstract}
		We prove a rigidity theorem for stationary and uniformly rotating solutions of the two-dimensional incompressible Euler equation with vorticity
		$
		\omega_0\in C^1(\mathbb{R}^2)\cap L^1(\mathbb{R}^2)
		\cap L^\infty(\mathbb{R}^2).
		$
		If the angular velocity $\Omega$ satisfies
		\[
\Omega\leq\frac12\inf_{\mathbb{R}^2}\omega_0
		\qquad\text{or}\qquad	\Omega\geq\frac12\sup_{\mathbb{R}^2}\omega_0,
		\]
		then $\omega_0$ is radially symmetric, in particular, every stationary vorticity of one sign in the class $C^1(\mathbb{R}^2)\cap L^1(\mathbb{R}^2)
		\cap L^\infty(\mathbb{R}^2)$ must be radially symmetric about some point. Furthermore, for $\Omega\neq0$, the center is necessarily the origin;
		
		The proof is based on the analysis for level-set geometry of a normalized stream function. A componentwise Bernoulli formula, together with the isoperimetric inequality and Pohozaev identities, yields a nonnegative defect measure encoding both geometric defects and topological branching. An analysis at infinity forces this measure to vanish, reducing the problem to a semilinear elliptic equation with a bounded nonnegative Borel nonlinearity. Radial symmetry then follows from a theorem of P.-L.\ Lions.

		\bigskip
		\noindent\textbf{Keywords}: the Euler equations, uniformly rotating
		vorticities, stationary vorticities, radial symmetry, level-set geometry.

		\bigskip
		
		\noindent\textbf{Mathematics Subject Classification (2020)}: Primary 35Q31; Secondary 76B47, 35B06.
	\end{abstract}

	\maketitle
	
	\tableofcontents

	\bibliographystyle{plain}

	\section{Introduction and main result}

	\subsection{The Euler equation}
	The two-dimensional incompressible Euler equation takes the following
	active-scalar form in vorticity variables:
	\begin{equation}\label{Euler}
\partial_t\omega+v\cdot\nabla\omega=0 \quad \text{in }\R_+\times\R^2,
	\end{equation}
	where $\omega$ denotes the scalar vorticity, and the velocity field $v$ is recovered from $\omega$ through the Biot--Savart law
\begin{equation}
 v(t,\cdot)=K*\omega(t,\cdot)\quad \text{with}\quad K(x)=\dfrac{1}{2\pi}\dfrac{(-x_2,x_1)}{|x|^2}.
\end{equation}    
 We refer to
\cite{Bedrossian2022,Majda2002,Marchioro1994,Wolibner1933,Yudovich1963}
	for the classical theory of \eqref{Euler}. A cornerstone of this theory is Yudovich's theorem \cite{Yudovich1963}, which establishes global existence and uniqueness of solutions with initial vorticities $\omega_0\in L^1(\mathbb{R}^2)\cap L^\infty(\mathbb{R}^2)$.
	
	We say that a vorticity is uniformly rotating with angular velocity
	$\Omega\in\mathbb{R}$ if it has the form
	\begin{equation}\label{1-2}
		\omega(t,x)=\omega_0(R_{-\Omega t}x),
	\end{equation}
	where $R_\vartheta$ denotes counterclockwise rotation through the angle $\vartheta$, i.e.,
    \begin{equation}
        R_\vartheta=
        \begin{pmatrix}
            \cos \vartheta & -\sin\vartheta\\
            \sin\vartheta & \cos\vartheta
        \end{pmatrix}.
    \end{equation}
	By the rotational covariance of the Biot--Savart law, the ansatz \eqref{1-2} defines a solution of \eqref{Euler} if and only if its profile $\omega_0$ satisfies
	\begin{equation}\label{1-3}
		\bigl(K*\omega_0+\Omega x^\perp\bigr)\cdot\nabla\omega_0=0  \qquad\text{in }\mathbb{R}^2,
	\end{equation}
 where $x^\perp=(x_2,-x_1)$.
	The stationary equation is recovered as the special case $\Omega=0$, i.e.
	\begin{equation*}
		(K*\omega_0)\cdot\nabla\omega_0=0  \qquad\text{in }\mathbb{R}^2.
	\end{equation*}
	
	In this paper, we investigate the rigidity of uniformly rotating solutions to \eqref{1-3}, including the stationary case $\Omega=0$. It is easy to check that every radial vorticity profile (i.e., $\omega_0(x)=\omega_0(|x|)$) satisfies \eqref{1-3} for every $\Omega$. Hence it is natural to ask under what additional hypotheses the solutions of \eqref{1-3} have to be radially symmetric. Results of this kind are usually referred to as \emph{rigidity results}, in contrast with \emph{flexibility results}, which investigate the existence of genuinely nontrivial solutions near some symmetric states.
	
	\subsection{Related work}
	There is by now an extensive literature on rigidity and flexibility for uniformly rotating vortex patches and smooth vorticities. Fraenkel \cite{Fraenkel2000}, using the method of moving planes, established the radial symmetry of stationary, simply connected vortex patches. Under the additional assumption that the patch is convex, Hmidi \cite{Hmidi2015} proved the corresponding rigidity result for simply connected rotating patches with $\Omega\in(-\infty,0)\cup\left\{\frac12\right\}$. G\'omez-Serrano, Park, Shi and Yao \cite{GomezSerrano2021} subsequently showed that every planar Euler vortex patch rotating with uniform angular  velocity $\Omega\in(-\infty,0]\cup\left[\frac12,\infty\right)$ is radial, without any convexity assumption, and that the two thresholds are sharp. In the smooth setting, they also proved radial symmetry for nonnegative, compactly supported vorticities rotating with uniformly velocity $\Omega\leq0$.
	
	More recently, Fan, Wang and Zhan \cite{Fan2025} proved that a compactly supported $C^2$ uniformly rotating vorticity is radial whenever
	$\Omega\leq\frac12\inf_{\mathbb{R}^2}\omega_0 $  or $\Omega\geq\frac12\sup_{\mathbb{R}^2}\omega_0$,
	allowing the vorticity to change sign. They also established the corresponding rigidity result for finite multipatches whose boundaries consist of finitely many pairwise disjoint Jordan curves. Analogous rigidity results for the Euler equation in a disk were obtained in \cite{Fan2024}.
	
	Regularity, quantitative estimates, and global bifurcation for vortex patches are treated in \cite{Bertozzi1993,Chemin1993,Serfati1994,Hmidi2015,Hmidi2013,deLaHoz2016,Hassainia2020,Park2022}. Nonradial rotating and time-periodic solutions in the complementary range of angular velocities are constructed in \cite{Garcia2020,Garcia2024,Enciso2025}. In particular, the smooth nonradial rotating flows of \cite{Enciso2025} rule out an unconditional rigidity theorem.
	
	Ruiz \cite{Ruiz2023} proved radial symmetry for compactly supported steady velocity fields under assumptions involving either the annular structure of the support or a nondegeneracy condition along its boundary. De Regibus, Esposito and Ruiz \cite{DeRegibus2026} considered entire stationary flows with finite kinetic energy. They showed that connectedness of the stagnation set leads to an autonomous semilinear formulation and, under additional hypotheses at infinity, to radial symmetry. Hamel and Nadirashvili \cite{Hamel2019,Hamel2023} established Liouville-type theorems based on the geometry of streamlines and the prescribed structure of the stagnation set. In bounded domains, Wang and Zhan \cite{Wang2023} derived rigidity from the topology of the stagnation set together with local symmetry arguments, while Gui, Xie and Xu \cite{Gui2026} classified slowly growing steady flows in terms of their sets of flow angles.
	
	Other structural results include the local description of steady states by Choffrut and Sz\'ekelyhidi \cite{Choffrut2014}, the flexibility and rigidity theory of Constantin, Drivas and Ginsberg \cite{Constantin2021}, stationary structures near Kolmogorov and Poiseuille flows \cite{CotiZelati2023}, and the analytic classification theorem of Elgindi, Huang, Said and Xie \cite{Elgindi2026a}. Elgindi and Huang \cite{Elgindi2026b} show that, even under additional structural assumptions, a smooth steady state need not admit a global single-valued relation between vorticity and stream function. At much weaker regularity, convex integration leads to a different flexibility theory \cite{DeLellis2009,DeLellis2012,DeLellis2013}. Sign-changing stationary vorticities are outside the rigidity class considered here: smooth nonradial compactly supported examples are known; see \cite{GomezSerrano2025,Enciso2026}.
	\subsection{Motivation and the main result}
	It is worth emphasizing that all of the rigidity results discussed above for smooth rotating vorticities assume compact support, and their proofs appear to rely crucially on this hypothesis. Consequently, they do not cover many natural rotating vortices with noncompact tails. On the other hand, as recalled earlier, Yudovich's theorem singles out
	$
	L^1(\mathbb{R}^2)\cap L^\infty(\mathbb{R}^2)$
	as the natural global well-posedness class for the vorticity equation. This leads to the question of whether the preceding rigidity theory continues to hold for classical rotating profiles at the Yudovich level, namely,
	$
	\omega_0\in C^1(\mathbb{R}^2) \cap L^1(\mathbb{R}^2) \cap L^\infty(\mathbb{R}^2).
	$
	Our main theorem answers this question affirmatively.
	\begin{theorem}\label{main}
		Let $\Omega\in\mathbb{R}$, and let
			$\omega_0\in C^1(\mathbb{R}^2) \cap L^1(\mathbb{R}^2) \cap L^\infty(\mathbb{R}^2)$
		satisfy
		\begin{equation}\label{1-6}
			\bigl(K*\omega_0+\Omega x^\perp\bigr)\cdot\nabla\omega_0=0  \qquad\text{in }\mathbb{R}^2.
		\end{equation}
		Suppose that
		\begin{equation}\label{threshold}
			\Omega\leq\frac12\inf_{\mathbb{R}^2}\omega_0 \qquad\text{or}\qquad \Omega\geq\frac12\sup_{\mathbb{R}^2}\omega_0.
		\end{equation}
		Then $\omega_0$ is radially symmetric. More precisely, there exist a point $x_0\in\mathbb{R}^2$ and a function $W\colon[0,\infty)\to\mathbb{R}$ such that
		\[
		\omega_0(x)=W(|x-x_0|)  \qquad\text{for every }x\in\mathbb{R}^2.
		\]
		If $\Omega\neq0$, the center $x_0$ must be chosen to be the origin.
	\end{theorem}

    There are a few remarks in order.

    \begin{remark}
	Theorem~\ref{main} gives a classification for rotating vorticities in the natural Yudovich framework. In particular, every stationary vorticity of one sign is radial about some point, whereas for a genuinely rotating solution the center of symmetry is pinned to the prescribed center of rotation, namely the origin. The theorem removes the compact-support hypothesis inherent in the preceding theory and extends rigidity from spatially localized equilibria to the full Yudovich setting.    
    \end{remark}

    \begin{remark}
	The threshold \eqref{threshold} admits a simple intrinsic elliptic interpretation. If $\psi$ denotes the standard stream function associated with $\omega_0$, then the relative stream function
	$\Psi=\psi+\frac{\Omega}{2}|x|^2$
	satisfies
	\begin{equation}\label{eq:effectivesource}
	-\Delta\Psi=\omega_0-2\Omega.
	\end{equation}
	Thus \eqref{threshold} is precisely the condition that the effective source in \eqref{eq:effectivesource}  has a fixed sign. The theorem therefore identifies this one-sided elliptic structure as the fundamental mechanism governing radial rigidity.
    \end{remark}

	\subsection{Outline of the proof}
	The proof is based on the level-set geometry of a normalized stream function, following ideas from elliptic symmetry theory; see \cite{Lions1981,Kesavan1994,Serra2013,Talenti1976,Zhan2026}. 
    
    The transport equation implies that the vorticity is constant on each connected regular level curve, but these constants may differ between components of the same level. Thus a global semilinear equation is not available at the outset. We first keep track of the individual components and then use the behavior at infinity to establish connectedness and symmetry.

Recall that $\Psi=\psi+\frac{\Omega}{2}|x|^2$ is the relative stream function. Excluding the trivial zero profile, we choose $\sigma\in\{-1,1\}$ according to \eqref{threshold} and set
\[
\Phi:=\sigma\Psi,\qquad \rho:=\sigma(\omega_0-2\Omega)\ge0.
\]
By defining $\psi$ through a renormalized logarithmic potential, we obtain
\[
-\Delta\Phi=\rho,\qquad \nabla^\perp\Phi\cdot\nabla\rho=0,\qquad \Phi(x)\to-\infty\quad\text{as }|x|\to\infty.
\]
Here $\Phi\in C^2(\R^2)$ and $\rho\in C^1(\R^2)\cap L^\infty(\R^2)$. In particular, the superlevel sets of $\Phi$ are bounded, even when the vorticity has noncompact support.

Set $\M:=\max_{\R^2}\Phi$. For $t<\M$, let $E_\Phi(t):=\{\Phi>t\}$ and denote its connected components by $\cC_\Phi(t)$. If $t$ is a regular value, meaning that $\nabla\Phi\ne0$ on $\Phi^{-1}(t)$, then $\cC_\Phi(t)$ is finite. Moreover, superharmonicity excludes holes, so each component $D$ is a Jordan domain bounded by a single regular level curve. The transport equation implies that $\rho$ is constant on $\partial D$; we denote this value by $c_D(t)$.

The same equation shows that $\rho\nabla\Phi$ is curl-free. Hence there exists a Bernoulli function $H$ satisfying $\nabla H=\rho\nabla\Phi$. Write $h_D(t)$ for its constant value on $\partial D$. By integrating along the level-set branches and then over the components, we obtain, for every regular value $t<\M$,
\[
I(t):=\sum_{D\in\cC_\Phi(t)}\int_D\bigl(H(x)-h_D(t)\bigr)\,\dd x=\int_t^\M A(s)\,\dd s,
\]
where $A(t):=\sum_{D\in\cC_\Phi(t)}c_D(t)|D|$ at regular values and $A(t):=0$ otherwise, with $|D|$ denoting the area of $D$. The main point is to justify the branchwise representation across critical levels, where components may split. This identity relates the variation of the boundary values to an integral of $H$.

We next introduce the mass-square function
\[
S(t):=\sum_{D\in\cC_\Phi(t)}\left(\int_D\rho(x)\,\dd x\right)^2.
\]
The divergence theorem expresses each component mass as the boundary integral of $|\nabla\Phi|$. By the coarea formula, the Cauchy--Schwarz inequality, and the planar isoperimetric inequality, we have $-S'(t)\ge8\pi A(t)$ on every interval of regular values. Since $S$ is nonincreasing across critical levels as well, we obtain the following identity of measures:
\[
-\dd S=8\pi A(t)\,\dd t+\mu,
\]
where $-\dd S$ is the negative distributional derivative of $S$, and $\mu$ is a nonnegative Radon measure on $(-\infty,\M)$. On regular levels, $\mu$ measures the excess in the geometric inequalities. At critical levels, it also accounts for splitting: separation into two components of positive limiting masses $m_1,m_2$ produces a loss of at least $2m_1m_2$ in $S$, and hence a positive atom of $\mu$. Since $S(t)\to0$ as $t\nearrow\M$, we have
\[
\delta(t):=S(t)-8\pi I(t)=\mu((t,\M))\ge0
\]
at every regular value. Thus $\delta$ is nonincreasing, and $\mu=0$ follows once $\delta(t_k)\to0$ along regular values $t_k\to-\infty$.

We now explain why $\mu=0$ implies radial symmetry. Every nonempty superlevel component has positive $\rho$-mass, so the absence of splitting forces every regular superlevel set to be connected. Its boundary is therefore a single Jordan curve, and the transport equation yields $\rho=f(\Phi)$ almost everywhere for a bounded nonnegative Borel function $f$ on $(-\infty,\M]$. Equality in the geometric inequalities further implies that $E_\Phi(t)$ is a disk for almost every regular value with $f(t)>0$. Applying the theorem of P.-L.\ Lions \cite{Lions1981} to $\Phi-t$ on these disks, we obtain radial symmetry and radial monotonicity; see also \cite[Theorem~1]{Serra2013}. Radial monotonicity and the nestedness of these superlevel sets then force a common center. Any remaining exterior region has $\rho=0$ and is treated by harmonic continuation.

It remains to prove that $\mu=0$. When $\Omega=0$, we have $\rho=\sigma\omega_0\ge0$ and $\rho\in L^1(\R^2)$. Set $\Gamma:=\int_{\R^2}\rho\,\dd x$. After normalizing $H$ to vanish at infinity, the componentwise formula gives $H\ge0$ and $\int_{\R^2}H\,\dd x=\int_{-\infty}^\M A(t)\,\dd t$. Since $S(t)\to\Gamma^2$ as $t\to-\infty$, the mass-square decomposition yields
\[
\Gamma^2=8\pi\int_{\R^2}H\,\dd x+\mu((-\infty,\M)).
\]
On the other hand, a global Pohozaev identity gives $\Gamma^2=8\pi\int_{\R^2}H\,\dd x$. To justify the passage to infinity, we subtract the radial logarithmic term $-\frac{\Gamma}{2\pi}\log|x|$ from $\Phi$ and show that the gradient energy of the remainder tends to zero along a sequence of expanding annuli. Comparing the two identities proves $\mu=0$.

When $\Omega\ne0$, we have
\[
\Phi=\varphi-\frac b4|x|^2,\qquad \rho=b+\sigma\omega_0,
\]
where $b:=2|\Omega|$, $\sigma=-\operatorname{sgn}\Omega$, and $\varphi:=\sigma\psi$. Since $\rho$ has infinite total mass, we instead estimate $\delta$ on low levels. The quadratic term and the decay $\nabla\varphi(x)\to0$ imply that sufficiently low level curves are radial graphs in large annuli. A boundary Pohozaev identity bounds $\delta(t)$ by the angular energy of $\varphi$ along these curves. Subtracting the radial logarithmic part leaves the angular derivative unchanged, while the gradient energy of the remainder tends to zero along suitable annuli. Averaging over the levels crossing these annuli gives $t_k\to-\infty$ with $\delta(t_k)\to0$. By the nonnegativity and monotonicity of $\delta$, we conclude that $\mu=0$. Finally, the asymptotic relation $\nabla\Phi(x)=-\frac b2x+o(1)$ fixes the center at the origin. In the stationary case, the center remains free.

	\subsection{Organization of the paper}
	The rest of the paper is organized as follows. Sections~\ref{s3}--\ref{s5} develop an abstract rigidity framework for a nonstandard elliptic system. Section~\ref{s6} proves Theorem~\ref{main} in the case $\Omega=0$, while Section~\ref{s7} treats the rotating case $\Omega\neq0$ and completes the proof. Appendix~\ref{app1} collects several preliminary estimates.

	\section{Finite-level geometry and the Bernoulli identity}\label{s3}
	In Sections \ref{s3}--\ref{s5}, we shall develop an abstract rigidity framework for a nonstandard elliptic system. This framework will subsequently be applied to the present problem to prove the main theorem. Assume $\Phi\in C^2(\R^2)$ and $\rho \in C^1(\R^2)\cap L^\infty(\R^2)$ satisfies
	\begin{equation}\label{3-1}
		\begin{aligned}
			\rho=-\Delta\Phi,  \qquad \nabla^\perp\Phi\cdot\nabla\rho=0,
		\end{aligned}
	\end{equation}
    together with
    \begin{equation}
        \rho\geq 0 \quad \text{and} \quad \Phi(x)\to-\infty \,\,\text{as }|x|\to\infty.\\
    \end{equation}
    It follows from \eqref{3-1} that one has
	\[
	\operatorname{curl}(\rho\nabla\Phi)=\nabla^\perp\Phi\cdot\nabla\rho=0.
	\]
	Hence there exists a globally defined  function $H\in C^1(\R^2)$, which is called Bernoulli function and is  unique up to an additive constant, such that
	\[
	\nabla H=\rho\nabla\Phi.
	\]
	The delicate point is to trace the corresponding branchwise constants across critical levels. This will lead to a one-dimensional representation of $H$ along each branch, which underlies the integral inequalities developed in the next section.
	
	Set
	\[
	\M:=\max_{\R^2}\Phi \quad \text{and}
	\quad 
	E_\Phi(t):=\{x\in\R^2:\Phi(x)>t\}\,\text{ for}\,\, t<\M.
	\]
Denote $\cC_\Phi(t)$ to be the collection of connected components of $E_\Phi(t)$. Furthermore, we denote $\mathcal{R}_\Phi$ to the range of the function $\Phi$ and 
\[
\mathscr R_\Phi:=\{t:  \nabla \Phi(x)\neq 0 \text{ for any } x\in \mathbb{R}^2 \text{ satisfying } \Phi(x)=t\}
\]
to be the set of regular values of $\Phi$. By Sard's theorem, $\mathcal{R}_\Phi\setminus\mathscr R_\Phi$ is a one-dimensional measure zero set.

	\subsection{Regular superlevels of $\Phi$}
	We begin with the following elementary description of the regular level sets of $\Phi$. The superharmonicity of $\Phi$ plays an important role here. Indeed, the strong maximum principle implies that no connected component of $E_\Phi(t)$ can have a hole.
	\begin{lemma}\label{le3-1}
		Let $t\in \mathscr R_\Phi$. Then
		\begin{enumerate}
			\item $\Phi^{-1}(t)$ is a finite disjoint union of $C^2$ Jordan curves;
        \item $\cC_\Phi(t)$ is finite;
			\item the map $D\mapsto\partial D$ is a bijection from $\cC_\Phi(t)$ onto the connected components of $\Phi^{-1}(t)$; for each $D\in\cC_\Phi(t)$, the set $D$ is the bounded Jordan domain determined by $\partial D$.
		\end{enumerate}
	\end{lemma}

\begin{proof}
Since $\Phi(x)\to-\infty$ as $|x|\to\infty$, the set $\Phi^{-1}(t)$ is compact. Moreover, $t$ is a regular value. By the implicit function theorem, $\Phi^{-1}(t)$ is a compact one-dimensional $C^2$ manifold without boundary, and hence a finite disjoint union of $C^2$ Jordan curves.

Let $\Gamma$ be any such curve, and let $U$ be the bounded Jordan domain enclosed by $\Gamma$. Since $\Delta\Phi\le0$ in $U$ and $\Phi=t$ on $\partial U$, the maximum principle implies that $\Phi\ge t$ in $U$. If $\Phi$ were constant in $U$, continuity would give $\nabla\Phi=0$ on $\Gamma$, contradicting the regularity of $t$. By the strong maximum principle, we therefore have $\Phi>t$ in $U$. Thus $U\subset E_\Phi(t)$. Moreover, $\partial U\cap E_\Phi(t)=\varnothing$, so $U$ is both open and closed in $E_\Phi(t)$. Since $U$ is connected, it is a connected component of $E_\Phi(t)$.

Conversely, let $D$ be a connected component of $E_\Phi(t)$. Since $D$ is bounded, we may choose $x\in\partial D\subset\Phi^{-1}(t)$. Let $\Gamma$ be the level curve through $x$, and let $U$ be its bounded Jordan domain. By the implicit function theorem, near $x$ the set $E_\Phi(t)$ lies precisely on the $U$-side of $\Gamma$. Hence $D\cap U\ne\varnothing$. Since both $D$ and $U$ are connected components of $E_\Phi(t)$, we have $D=U$. This proves all the assertions.
\end{proof}

	The next result provides a finer geometric description of regular superlevels.
	
	\begin{lemma}\label{le3-2}
		The set $\mathscr R_\Phi$ is open. If $[a,b]\subset\mathscr R_\Phi$, then $\Phi^{-1}([a,b])$ is a finite disjoint union of cylinders, each diffeomorphic to $\mathbb S^1\times[a,b]$ in such a way that $\Phi$ corresponds to projection onto the second factor. Moreover, whenever $a\le s<t\le b$, each connected component of $E_\Phi(s)$ contains exactly one connected component of $E_\Phi(t)$.
	\end{lemma}

	\begin{proof}
    The proof is divided into three steps.

    {\it Step 1.} We first show that $\mathscr R_\Phi$ is open. Fix $t_0\in\mathscr R_\Phi$ and suppose, to the contrary, that $t_0$ is not an interior point of $\mathscr R_\Phi$. Then there exist critical points $\{x_k\}$ such that $\Phi(x_k)\to t_0$ as $k\to\infty$. Since $\Phi(x)\to-\infty$ as $|x|\to\infty$, the sequence $\{x_k\}$ is bounded. Passing to a subsequence, we may assume that $x_k\to x_0$ as $k\to \infty$. By continuity, we have
		\[
		\Phi(x_0)=t_0 \quad\text{and}\quad  \nabla\Phi(x_0)=0.
		\]
	This contradicts the fact that $t_0$ is a regular value.

        {\it Step 2.}
		Denote
		$\mathcal{K}:=\Phi^{-1}([a,b])$.
		If $\mathcal K=\varnothing$, there is nothing to prove. Otherwise, $\mathcal K$ is closed by continuity of $\Phi$ and bounded because $\Phi(x)\to-\infty$ as $|x|\to\infty$. Hence $\mathcal K$ is compact. Since $[a,b]\subset\mathscr R_\Phi$, we have
     \[
     \min_{\mathcal K}|\nabla\Phi|>0.
     \]
		Hence there exists an $\varepsilon>0$ such that for any $x\in \mathcal{K}_\varepsilon:=\left\{x:dist(x,\mathcal{K})<\varepsilon\right\}$, $\nabla\Psi(x)>0$. define
		\begin{equation}
        \left\{
        \begin{aligned}
		&\frac{d\Theta}{d\tau} (\tau, x)=\frac{\nabla\Phi}{|\nabla\Phi|^2}(\Theta(\tau, x)),\\
        & \Theta(0, x) =x\in \mathcal{K}.
        \end{aligned}
        \right.
        \end{equation}
		Clearly, one has
		\[
		\frac{\dd}{\dd\tau}\Phi(\Theta(\tau, x))=1.
		\]
		Compactness of $\mathcal{K}$ ensures that the flow is defined for all times needed to cross the strip $\mathcal{K}$. Consequently, the map $F(x,r):=\Theta(r-a, x)$ 
		is a $C^1$ diffeomorphism from $\Phi^{-1}(a)\times[a,b]$ to $\mathcal{K}$, whose inverse is 
		$F^{-1}(z)=\bigl(\Theta({a-\Phi(z)},z),\Phi(z)\bigr)$.
		Since $\Phi^{-1}(a)$ is a finite disjoint union of $C^2$ Jordan curves, $\mathcal{K}$ is a finite disjoint union of cylinders diffeomorphic to $\mathbb S^1\times[a,b]$.

         {\it Step 3.}
		Now fix $a\le s<t\le b$ and denote
			$F_\Phi(t):=\{\Phi\ge t\}$.
		For $x\in E_\Phi(s)$ and $\lambda\in[0,1]$, define
		\[
		\vartheta(\lambda, x):=
		\begin{cases}
			\Theta(\lambda(t-\Phi(x)), x),
			&s<\Phi(x)<t,\\[1mm]
			x,
			&\Phi(x)\ge t.
		\end{cases}
		\]
		If $s<\Phi(x)<t$, then
		\[
		\Phi(\vartheta(\lambda, x))=(1-\lambda)\Phi(x)+\lambda t.
		\]
		Thus $\vartheta$ is a strong deformation retraction of $E_\Phi(s)$ onto $F_\Phi(t)$. It follows that $E_\Phi(s)$ and $F_\Phi(t)$ have the same connected components. Since $t$ is a regular value, the components of $F_\Phi(t)$ are precisely the closures of the components of $E_\Phi(t)$. Hence the inclusion $E_\Phi(t)\subset E_\Phi(s)$ induces a bijection between their connected components. The proof is thus complete.
	\end{proof}

	\subsection{The component formula for $H$}
	
	Note that if $t\in \mathscr R_\Phi$ and $D\in\cC_\Phi(t)$, then $\rho$ and $H$ are constant on $\partial D$.  We denote the constants by
	\begin{equation}\label{3-4}
		c_D(t):=\rho|_{\partial D},\qquad h_D(t):=H|_{\partial D}.
	\end{equation}
	For $x\in E_\Phi(t)$, let $D_t(x)$ denote the connected component of $E_\Phi(t)$ containing $x$. If $ s_1<s_2<\Phi(x)$, then $D_{s_2}(x)\subset D_{s_1}(x)$.
	
	We next derive a one-dimensional representation of $H$ along each branch. We begin with two auxiliary results.
	\begin{lemma}\label{lem3-3}
		Let $a$ be a regular value, $D\in\cC_\Phi(a)$, and $x,y\in D$. Set
		\[
		m(x,y):=\min\{\Phi(x),\Phi(y)\}
		\]
		and
		\[
		\Lambda(x,y):= \{s\in(a,m(x,y)):D_s(x)=D_s(y)\}.
		\]
		Then $\Lambda(x,y)$ is nonempty and  if $s_2\in\Lambda(x,y)$ and $a<s_1<s_2$, then $s_1\in \Lambda(x,y)$. In particular,
		$
		\lambda(x,y):=\sup\Lambda(x,y)\in(a,m(x,y)]
		$
		is well defined. Moreover, every rectifiable curve $\Gamma:[0,1]\to D$ joining $x$ to $y$ satisfies
		\begin{equation}\label{3-5}
			\min_{0\le\tau\le1}\Phi(\Gamma(\tau))\le\lambda(x,y)
		\end{equation}
		and
		\begin{equation}\label{3-6}
			\Phi(x)+\Phi(y)-2\lambda(x,y) \le \operatorname{Var}(\Phi\circ\Gamma).
		\end{equation}
	\end{lemma}
	
	\begin{proof}
		Choose a continuous path $\Gamma_0\subset D$ from $x$ to $y$, and set 	$\mu_0:=\min_{0\le\tau\le1}\Phi(\Gamma_0(\tau)).$
		
		The compactness of $\Gamma_0([0,1])$ and the inclusion
		$\Gamma_0([0,1])\subset D\subset E_\Phi(a)$ imply that $\mu_0>a$. Hence one has
		$
		(a,\min\{\mu_0,m(x,y)\})\subset\Lambda(x,y),
		$
		Thus $\Lambda(x,y)$ is nonempty. If $s_2\in\Lambda(x,y)$ and $a<s_1<s_2$, the common component of $E_\Phi(s_2)$ containing $x$ and $y$ is a connected subset of $E_\Phi(s_1)$, and therefore $s_1\in\Lambda(x,y)$.
		
		Now let $\Gamma$ be a rectifiable curve from $x$ to $y$, and set $\varsigma:=\min_{0\le\tau\le1}\Phi(\Gamma(\tau)).$
		If $\varsigma>\lambda(x,y)$, then any $s\in(\lambda(x,y),\varsigma)$ satisfies $\Gamma\subset E_\Phi(s)$, contrary to the definition of $\lambda(x,y)$. Thus $\varsigma\le\lambda(x,y)$, which proves
		\eqref{3-5}.
		
		If $\tau_0$ is such that $\Phi(\Gamma(\tau_0))=\varsigma$, then
		\begin{align*}
			\operatorname{Var}(\Phi\circ\Gamma)
			&\ge \bigl|\Phi(x)-\varsigma\bigr|+\bigl|\Phi(y)-\varsigma\bigr|\\
			&=\Phi(x)+\Phi(y)-2\varsigma\\
			&\ge\Phi(x)+\Phi(y)-2\lambda(x,y),
		\end{align*}
		which is \eqref{3-6}. The proof is thus completed.
	\end{proof}

	\begin{lemma}\label{le3-4}
		Define
		\[
		\mathcal O:=\bigl\{(x,s)\in\R^2\times\R: s\in\mathscr R_\Phi,\ s<\Phi(x)\bigr\}
		\]
		and
		\begin{equation}\label{3-8}
			\mathfrak c(x,s):=
			\begin{cases}
				c_{D_s(x)}(s),&(x,s)\in\mathcal O,\\
				0,&(x,s)\notin\mathcal O.
			\end{cases}
		\end{equation}
		Then $\mathfrak c$ is continuous on $\mathcal O$ and Borel measurable
		on $\R^2\times\R$. In particular, for every fixed $x\in\R^2$, the map
		\[
		s\mapsto c_{D_s(x)}(s),
		\]
		extended by zero at critical values, is Borel measurable on $(-\infty,\Phi(x))$.
	\end{lemma}
	
	\begin{proof}
		Since $\mathscr R_\Phi$ is open, so is $\mathcal O$. Fix
		$(x_0,s_0)\in\mathcal O$, and choose a ball $B$ such that $x_0\in$
		$
		x_0\in B\Subset D_{s_0}(x_0).
		$
		After shrinking $\delta>0$, we may assume that
		\[
		J:=[s_0-\delta,s_0+\delta]\subset\mathscr R_\Phi \qquad\text{and}\qquad s_0+\delta<\inf_{\overline B}\Phi.
		\]
		Thus $B\subset E_\Phi(s)$ for every $s\in J$. By Lemma \ref{le3-2}, the component $D_{s_0}(x_0)$ extends
		uniquely to a nested family
		$
		D(s)\in\cC_\Phi(s)$, $s\in J,
		$
		and the component correspondence across the regular band implies that
		$
		B\subset D(s)\text{ for every }s\in J. 
		$
		In particular,
		$
		D_s(x)=D(s)\text{ for }x\in B,\ s\in J.
		$
        
		Choose $z_0\in\partial D(s_0)$. The normalized gradient flow used in the proof of Lemma \ref{le3-2} gives a $C^1$ curve $z:J\to\R^2$ such that
		$z(s)\in\partial D(s)$ and $ z(s_0)=z_0.$
		The fact that $\rho$ is constant on each regular level curve implies
		$c_{D(s)}(s)=\rho(z(s)), s\in J.$
		Hence, we have $\mathfrak c(x,s)=\rho(z(s))$ on $B\times(s_0-\delta,s_0+\delta)$, 
		which is continuous. Thus $\mathfrak c$ is continuous on $\mathcal O$. Its extension by zero to the complement of the open set $\mathcal O$ is therefore Borel measurable.
	\end{proof}
	
	We are now ready to give a one-dimensional representation of $H$ along each branch.

	\begin{proposition}\label{Pro3-5}
		Let $a$ be a regular value and let $D\in\cC_\Phi(a)$. Then, for every $x\in D$,
		\begin{equation}\label{3-15}
			H(x)-h_D(a)=\int_a^{\Phi(x)}c_{D_s(x)}(s)\,\dd s,
		\end{equation}
		where the integrand is extended by zero at critical values.
		
		In particular, let $a<b$ be regular values, and suppose that $D_b\subset D_a$, where $D_a\in\cC_\Phi(a)$ and $D_b\in\cC_\Phi(b)$. Then
		\begin{equation}\label{3-16}
			h_{D_b}(b)-h_{D_a}(a)=\int_a^b c_{D_s}(s)\,\dd s,
		\end{equation}
		where, for each regular value $s\in(a,b)$, $D_s$ is the unique connected component of $E_\Phi(s)$ containing $D_b$.
		
		Moreover,
		\begin{equation}\label{3-17}
			0\le h_{D_b}(b)-h_{D_a}(a) \le \|\rho\|_{L^\infty(D_a)}(b-a).
		\end{equation}
	\end{proposition}
	
	\begin{proof}
We divide the proof into four steps.
    
		{\it Step 1.} Set up. $M:=\|\rho\|_{L^\infty(D)}$. For $x\in D$ and $a<s<\Phi(x)$, write $c_x(s):=\mathfrak c(x,s)$. By Lemma \ref{le3-4}, $c_x$ is Borel measurable. Since every component $D_s(x)$ is contained in $D$, we have $0\le c_x(s)\le M$.
		Define
		\begin{equation}\label{3-18}
			F(x):=h_D(a)+\int_a^{\Phi(x)}c_x(s)\,\dd s,\qquad x\in D.
		\end{equation}
		The estimate
		\begin{equation}\label{3-19}
			|F(x)-h_D(a)|\le M\bigl(\Phi(x)-a\bigr)
		\end{equation}
		shows that $F$ extends continuously to $\partial D$ by setting $F=h_D(a)$ there.
		
		{\it Step 2.} Lipschitz estimate. Let $x,y\in D$ and put $\lambda:=\lambda(x,y)$ as in Lemma \ref{lem3-3}, we have
		$D_s(x)=D_s(y)$ $\text{for every regular }s\in(a,\lambda)$ and hence $c_x(s)=c_y(s)$ $\text{for a.e. }s\in(a,\lambda)$ by Sard's theorem.
		 Which implies
        \begin{equation*}
		\begin{aligned}
|F(x)-F(y)|&=\left|\int_\lambda^{\Phi(x)}c_x(s)\dd s-\int_\lambda^{\Phi(y)}c_y(s)\dd s\right|\\
        &\le M\bigl(\Phi(x)-\lambda\bigr) + M\bigl(\Phi(y)-\lambda\bigr).
	\end{aligned}
    \end{equation*}
		By Lemma \ref{lem3-3}, for every rectifiable curve $\Gamma\subset D$ joining $x$ to $y$, we have
		\begin{equation}\label{3-20}
			|F(x)-F(y)|\le M\operatorname{Var}(\Phi\circ\Gamma) \le M\int_\Gamma|\nabla\Phi|\,\dd s.
		\end{equation}
		Taking $\Gamma$ to be a line segment on balls compactly contained in $D$, we see that $F$ is locally Lipschitz.
        
        {\it Step 3.} Gradient estimate. Applying \eqref{3-20} to coordinate segments and passing to difference quotients yields
		\begin{equation}\label{3-21}
			|\partial_iF| \le M|\partial_i\Phi| \qquad\text{a.e. in }D.
		\end{equation}
		
		We next identify the gradient of $F$ at points lying on regular
		levels. Fix $x_0\in D$ such that $s_0:=\Phi(x_0)$ is a regular value.
		Choose $\eta>0$ so that
		\[
		a<s_0-2\eta, \qquad [s_0-2\eta,s_0+2\eta]\subset\mathscr R_\Phi.
		\]
		Let $Q$ be the component of 
		$\Phi^{-1}\bigl((s_0-2\eta,s_0+2\eta)\bigr)$
		containing $x_0$, and set
		$Q_0:=Q\cap\{s_0-\eta<\Phi<s_0+\eta\}.$
		By Lemma \ref{le3-2}, the level curves
		$\Gamma_r:=Q\cap\Phi^{-1}(r)$
		form a single branch, and
		$c(r):=\rho|_{\Gamma_r}$
		is continuous. Since $Q_0$ lies in one component of $E_\Phi(s_0-\eta)$, the part of \eqref{3-18} below $s_0-\eta$ is independent of $x\in Q_0$. Thus
		\[
		F(x)=C+\int_{s_0-\eta}^{\Phi(x)}c(r)\,\dd r, \qquad x\in Q_0,
		\]
		for some constant $C$, and hence
		\begin{equation}\label{3-22}
			\nabla F(x)=c(\Phi(x))\nabla\Phi(x)=\rho(x)\nabla\Phi(x), \qquad x\in Q_0.
		\end{equation}
		
		Let $N$ be the set of critical values of $\Phi$. Since $|N|=0$, we have $\nabla\Phi=0$ almost everywhere on $\Phi^{-1}(N)\cap D$; see \cite[Theorem 6.19]{Lieb2001}.
		By \eqref{3-21}, $\nabla F=0$ almost everywhere there as well. Combining this with \eqref{3-22}, yields
		\[
		\nabla F=\rho\nabla\Phi=\nabla H \qquad\text{a.e. in }D.
		\]
		Since $D$ is connected, $F-H$ is constant in $D$. Both $F$ and $H$ are continuous, and
		\[
		F(x),\,H(x)\to h_D(a)\qquad\text{as }x\to\partial D,\ x\in D,
		\]
		by \eqref{3-19} and \eqref{3-4}. Hence $F=H$ in $D$, proving \eqref{3-15}.
		
		Finally, let $D_b\subset D_a$ and choose $x\in\partial D_b$. Since
		$\Phi(x)=b>a$, we have $x\in D_a$. Applying
		\eqref{3-15} in $D_a$, we obtain
		\[
		h_{D_b}(b)-h_{D_a}(a)=\int_a^b c_{D_s}(s)\,\dd s.
		\]
		The bound \eqref{3-17} follows immediately from
		\[
		0\le c_{D_s}(s)\le\|\rho\|_{L^\infty(D_a)}.
		\]
		The proof is thus completed.
	\end{proof}
	
	Integrating the branchwise representation in Proposition \ref{Pro3-5} over the components of a superlevel set leads to the following important identity, which will be used in the integral estimates of the next section.
	
	\begin{corollary}\label{cor3-6}
		Let $t<\M$ be a regular value, and define
		\begin{equation}\label{3-23}
			I(t):=\sum_{D\in\cC_\Phi(t)}\int_D\bigl(H(x)-h_D(t)\bigr)\,\dd x.
		\end{equation}
		Then $I(t)\ge0$, and
		\begin{equation}\label{3-24}
			I(t)=\int_t^\M\sum_{D\in\cC_\Phi(s)}c_D(s)|D|\,\dd s,
		\end{equation}
		where the integrand on the right is understood to be zero at critical values.
	\end{corollary}
	
	\begin{proof}
		By Proposition \ref{Pro3-5}, for $D\in\cC_\Phi(t)$ and $x\in D$, we have
		\begin{equation*}
			H(x)-h_D(t)=\int_t^{\Phi(x)}\mathfrak c(x,s)\,\dd s.
		\end{equation*}
		Since $\rho\ge0$, the integrand is nonnegative, and hence $I(t)\ge0$. The function
		\[
		(x,s)\mapsto \mathfrak c(x,s)\1_{\{t<s<\Phi(x)\}}
		\]
		is Borel measurable and nonnegative. Summing over the components of $E_\Phi(t)$ and applying Tonelli's theorem, yield
		\begin{equation*}
			I(t)=\int_{E_\Phi(t)}\int_t^{\Phi(x)} \mathfrak c(x,s)\,\dd s\,\dd x=\int_t^\M\int_{E_\Phi(s)}\mathfrak c(x,s)\,\dd x\,\dd s.
		\end{equation*}
		For every regular value $s$, the function $\mathfrak c(\,\cdot\,,s)$ is equal to $c_D(s)$ on each component $D\in\cC_\Phi(s)$. Therefore,
		\[
		\int_{E_\Phi(s)}\mathfrak c(x,s)\,\dd x=\sum_{D\in\cC_\Phi(s)}c_D(s)|D|.
		\]
		At critical values, $\mathfrak c$ vanishes by definition. Substitution into the preceding identity yields \eqref{3-24}. This finishes the proof of the corollary.
	\end{proof}
	
	\section{The mass-square inequality}\label{s4}
	In this section, we establish a differential inequality for the mass-square function along regular level-set branches. Together with the integral representation in Corollary \ref{cor3-6}, this leads to the mass-square inequality underlying the rigidity argument.
	
	\subsection{Differentiation along a regular component}
	Let $J$ be an interval consisting entirely of regular values of $\Phi$, and let $\{D(s)\}_{s\in J}$ be one of the families of components provided by Lemma \ref{le3-2}. For $s\in J$, set
	\begin{equation}\label{4-0}
		a(s):=|D(s)|, \qquad m(s):=\int_{D(s)}\rho(x)\dd x,
		\qquad  c(s):=\rho|_{\partial D(s)}.
	\end{equation}
	The divergence theorem and the standard differentiation formulas for
	regular superlevel sets imply
	\begin{align}
		m(s)
		&=\int_{\partial D(s)}|\nabla\Phi|\,\dd\cH^1,\label{4-1}\\
		-a'(s) &=\int_{\partial D(s)}\frac{1}{|\nabla\Phi|}\,\dd\cH^1, \label{4-2}\\
		m'(s) &=c(s)a'(s),
		\label{4-3}
	\end{align}
	for every $s\in J$. Indeed, the outward unit normal to $D(s)$ is $-\nabla\Phi/|\nabla\Phi|$ on $\partial D(s)$, while \eqref{4-2} and \eqref{4-3} follow from the coarea formula.
	
	\begin{lemma}\label{le4-1}
		Along every regular component family, it holds that
		\begin{equation}\label{4-4}
			-\frac{\dd}{\dd s}m^2(s)\ge 8\pi c(s)a(s).
		\end{equation}
		If $c(s)>0$, equality holds if and only if $D(s)$ is a disk and $|\nabla\Phi|$ is constant on $\partial D(s)$.
	\end{lemma}
	
	\begin{proof}
		By \eqref{4-1}, \eqref{4-2} and the Cauchy--Schwarz inequality, we have
		\[
		m(s)(-a'(s))=\left(\int_{\partial D(s)}|\nabla\Phi|\,\dd\cH^1\right) \left(\int_{\partial D(s)}\frac1{|\nabla\Phi|}\,\dd\cH^1\right)
		\ge \cH^1(\partial D(s))^2.
		\]
		The planar isoperimetric inequality yields
		\[
		\cH^1(\partial D(s))^2\ge4\pi a(s).
		\]
		In view of \eqref{4-3}, we obtain
		\begin{equation}\label{dsm2}
		-\frac{\dd}{\dd s}m^2(s)=-2m(s)m'(s)= 2c(s)m(s)(-a'(s))
		\ge 8\pi c(s)a(s).
		\end{equation}
		If $c(s)>0$, equality in \eqref{dsm2} holds if and only if both the Cauchy--Schwarz and isoperimetric inequalities are corresponding equalities. Hence $|\nabla\Phi|$ is constant on $\partial D(s)$ and $D(s)$ is a disk. The converse is immediate. This finishes the proof of the lemma.
	\end{proof}
	
	\subsection{The mass-square function}
	For arbitrary $t<\M$, the open set $E_\Phi(t)$ has at most countably many components. Define
	\begin{equation}\label{4-7}
		S(t):=\sum_{D\in\cC_\Phi(t)}  \left(\int_D\rho(x)\dd x\right)^2.
	\end{equation}
	The series is finite because
	\[
	S(t)\le\left(\int_{E_\Phi(t)}\rho\right)^2<\infty.
	\]

	\begin{lemma}\label{le4-2}
		The function $S$ is nonincreasing, and
		\begin{equation}\label{4-8}
			\lim_{t\nearrow\M}S(t)=0.
		\end{equation}
	\end{lemma}
	
	\begin{proof}
		Let $s<t$. Each component of $E_\Phi(t)$ lies in a unique ancestor component of $E_\Phi(s)$. Fix an ancestor of mass $M$, and let $m_1,m_2,\dots$ be the masses of its descendants at level $t$.  The descendants are disjoint subsets of the ancestor, so $\sum_jm_j\le M$, and hence
		\[
		\sum_jm_j^2\le\left(\sum_jm_j\right)^2\le M^2.
		\]
		Summing over all ancestors gives $S(t)\le S(s)$. The sets $E_\Phi(t)$ decrease as $t$ increasingly converge to
		$K:=\{\Phi=\M\}$.
		At every point of $K$, $\nabla\Phi=0$, and hence $\rho=-\Delta\Phi=0$ almost everywhere on $K$. It follows that
		\[
		\lim_{t\nearrow M_\Phi}\int_{E_\Phi(t)}\rho\dd x=\int_K\rho\dd x=0.
		\]
		Since $S(t)\le(\int_{E_\Phi(t)}\rho)^2$, \eqref{4-8} follows. The proof is thus completed.
	\end{proof}

	For each regular value $t$, define
	\begin{equation}\label{4-9}
		A(t):=\sum_{D\in\cC_\Phi(t)}c_D(t)|D|,
	\end{equation}
	and set $A(t):=0$ at critical values. By the definition of $\mathfrak c$, we have
	\begin{equation}\label{4-10}
		A(t)=\int_{\R^2}\mathfrak c(x,t)\,\dd x.
	\end{equation}
	It follows from Lemma \ref{le3-4} that $A$ is Borel measurable. Moreover, Corollary \ref{cor3-6} gives
	\[ I(t)=\int_t^\M A(s)\,\dd s<\infty \] for every regular value $t<\M$. Since regular values are dense, it follows that  $A\in L^1_{\loc}((-\infty,\M))$.

	\begin{lemma}\label{le4-3}
		Let $J\subset\mathscr R_\Phi$ be an open interval. Then
		\begin{equation}\label{4-11}
			S\in C^1(J),\ \ \ A\in C(J),\ \ \ -S'(t)\ge8\pi A(t) \quad\text{for every }t\in J.
		\end{equation}
	\end{lemma}
	
	\begin{proof}
		By Lemma \ref{le3-2}, the components of $E_\Phi(t)$ can be labeled
		consistently on $J$ as\\
		$D_1(t)$,...,$D_N(t)$.
		For each branch, write
		\[
		a_j(t):=|D_j(t)|, \qquad m_j(t):=\int_{D_j(t)}\rho(x)\,\dd x,
		\qquad c_j(t):=\rho|_{\partial D_j(t)}.
		\]
		By \eqref{4-1}--\eqref{4-3}, we have $a_j,m_j\in C^1(J)$ and $c_j\in C(J)$. Consequently,
		\[
		S(t)=\sum_{j=1}^N m^2_j(t) \qquad\text{and}\qquad A(t)=\sum_{j=1}^N c_j(t)a_j(t)
		\]
		belong to $C^1(J)$ and $C(J)$, respectively. In view of \eqref{4-4}, we get
		\[
		-S'(t)=-\sum_{j=1}^N\frac{\dd}{\dd t}m^2_j(t)
		\ge 8\pi\sum_{j=1}^Nc_j(t)a_j(t)
		=8\pi A(t).
		\]
		The proof of the lemma is thus completed.
	\end{proof}

	The next proposition extends the differential inequality in \eqref{4-11} across critical levels by interpreting the variation of $S$ as a Stieltjes measure and recording the resulting defect in a nonnegative Radon measure.
	\begin{proposition}\label{Pro4-4}
		After replacing $S$ at its discontinuity points by its right-continuous representative, there is a nonnegative Radon measure $\mu$ on $(-\infty,\M)$ such that
		\begin{equation}\label{4-13}
			-\dd S=8\pi A(t)\dd t+\mu.
		\end{equation}
		For every regular value $t$,
		\begin{equation}\label{4-14}
			\delta(t):=S(t)-8\pi I(t)=\mu((t,\M))\ge0.
		\end{equation}
		Thus $\delta$ has a nonincreasing representative.
	\end{proposition}
	
	\begin{proof}
	 For a nonincreasing right-continuous function $S$, let
		$\nu:=-DS$
		be its distributional Stieltjes measure.  Then, for $\alpha<\beta$, we have
		\begin{equation}\label{4-15}
			\nu((\alpha,\beta])=S(\alpha)-S(\beta),\qquad \nu(\{t\})=S(t^-)-S(t).
		\end{equation}
		At a continuity point, $S(t)=S(t^-)=S(t^+)$. The original geometric function \eqref{4-7} is continuous at every regular value by Lemma \ref{le4-3}, so passing to the right-continuous representative changes no value used on a regular level.
		
		By Lemma \ref{le4-3}, on each connected component of $\mathscr R_\Phi$, we have
		\[
		\nu=-S'(t)\dd t
		\ \ \ \text{and}\ \ \
		-S'(t)\ge8\pi A(t).
		\]
		Consequently, for every Borel set $B\subset\mathscr R_\Phi$, it holds that
		\[
		\nu(B)=\int_B-S'(t)\dd t \ge8\pi\int_BA(t)\dd t.
		\]
		Since the critical-value set has Lebesgue measure zero and $A=0$ there, the measure
		\begin{equation}\label{4-16}
			\mu:=\nu-8\pi A(t)\dd t
		\end{equation}
		is nonnegative on every Borel set: on the regular set this follows from the preceding inequality, and on its complement only the nonnegative measure $\nu$ remains. This proves \eqref{4-13}.
		
		Let $t$ be regular, and choose regular values $\left\{b_k\right\}$ which increasingly converge to $\M$. Since both $t$ and $b_k$ are continuity points of $S$, it follows from \eqref{4-15} that
		$\nu((t,b_k])=S(t)-S(b_k)$.
		The intervals $(t,b_k]$ increase to $(t,\M)$. Hence, by \eqref{4-8} and continuity from below of $\nu$, we obtain
		\begin{equation}\label{4-17}
			\nu((t,\M))=S(t).
		\end{equation}
		It follows from the measure inequality $\nu\ge8\pi A(t)\dd t$ that
		\[
		8\pi\int_t^\M A(s)\dd s\le S(t).
		\]
		In view of Corollary \ref{cor3-6}, we have
		\begin{equation}\label{4-18-0}
			I(t)=\int_t^\M A(s)\dd s.
		\end{equation}
		Combining \eqref{4-16}, \eqref{4-17} and \eqref{4-18-0}, we get \eqref{4-14}. The function $t\mapsto\mu((t,\M))$ is obviously nonincreasing. The proof is thus complete.
	\end{proof}
	
	\subsection{Pohozaev identity}\label{s3-3}
	We next establish a Pohozaev-type identity, analogous to the classical one for semilinear elliptic equations.
	\begin{proposition}\label{Pro4-5}
		For every $t\in \mathscr R_\Phi$, we have
		\begin{equation}\label{4-18}
			I(t)=\frac14\sum_{D\in\cC_\Phi(t)} \int_{\partial D}(x\cdot n)|\nabla\Phi|^2\dd\cH^1.
		\end{equation}
	\end{proposition}
	
	\begin{proof}
		Fix $D\in\cC_\Phi(t)$, and write $h=h_D(t)$. 
		 The straightforward calculations give
		\[
		\Div\left(\nabla\Phi\otimes\nabla\Phi-\frac12|\nabla\Phi|^2I\right)=(\Delta\Phi)\nabla\Phi.
		\]
        We introduce the following stress--energy tensor, which is the local
		tensorial form underlying the Rellich--Pohozaev identity:
		\[
		T=\nabla\Phi\otimes\nabla\Phi -\frac12|\nabla\Phi|^2I+(H-h)I.
		\]
		Using $\Delta\Phi=-\rho$ and $\nabla H=\rho\nabla\Phi$, one obtains $\Div T=0$.  In dimension two, we have
		\[
		\operatorname{tr}T=2(H-h).
		\]
		Hence
		\[
		\Div(Tx)=2(H-h).
		\]
		Integrating over $D$, we get
		\[
		2\int_D(H-h)\dd x=\int_{\partial D}(Tx)\cdot n\dd\cH^1.
		\]
		On $\partial D$, $H-h=0$ and $\nabla\Phi=-|\nabla\Phi|n$, so
		\[
		(Tx)\cdot n=\frac12(x\cdot n)|\nabla\Phi|^2,
		\]
		from which \eqref{4-18} follows. The proof is thus completed.
	\end{proof}
	
	\section{Rigidity via the vanishing of $\mu$}\label{s5}
	In this section, we prove that the vanishing of $\mu$ forces both $\Phi$ and $\rho$ to be radially symmetric.

	\subsection{Connectedness of regular superlevels} We begin by showing that
	$E_\Phi(t)$ is connected for every regular value $t<\M$. We first note that, for every nonempty component $D$ of $E_\Phi(t)$, it holds that
	\begin{equation}\label{5-1}
		\int_D\rho(x)\,\dd x>0.
	\end{equation}
	Indeed, otherwise $\rho\equiv0$ in $D$, and $\Phi$ would be harmonic there, contradicting the maximum principle since $\Phi=t$ on $\partial D$ and $\Phi>t$ in $D$.
	
	\begin{lemma}\label{le5-1}
		Let $t_*<\M$. Assume that there exist regular values $\left\{s_k\right\}$ and $\left\{r_k\right\}$ satisfying
		\begin{equation}\label{rbj}
		\lim_{k\to\infty}s_k=\lim_{k\to\infty}r_k= t_*, \text{ and}\quad s_k<s_{k+1}<t_*<r_{k+1}<r_k.
	    \end{equation}
		Suppose that
		$M_k\in\cC(s_k)$,  $D_{1,k},D_{2,k}\in\cC(r_k)$, and $D_{1,k}\ne D_{2,k}$,
		such that
		$D_{1,k}\cup D_{2,k}\subset M_k$.
		Assume moreover that
		$M_{k+1}\subset M_k$, $D_{i,k}\subset D_{i,k+1}$ for $i=1,2$,
		and that
		$\inf_k\int_{D_{i,k}}\rho\,\dd x>0$, $i=1,2$.
		Then $-\dd S$ has an atom at $t_*$. More precisely, if $m_i:=\lim_{k\to\infty}\int_{D_{i,k}}\rho\,\dd x>0$,
		then
		\begin{equation}\label{5-2}
			(-\dd S)(\{t_*\})\ge2m_1m_2 .
		\end{equation}
	\end{lemma}
	
	\begin{proof}
		By the nestedness and $\rho\ge0$, the limits defining $m_i$ exist.
		For each $k$, let
		$\left\{m_{i,k}\right\}$
		be the masses of the descendant components of $E_{r_k}$ contained in $M_k$, with the first two corresponding to $D_{1,k}$ and $D_{2,k}$. Define
		$\mathfrak{m}_k(\rho):=\int_{M_k}\rho\,\dd x$,
		we have
		\[
		\sum_j m_{j,k}\le \mathfrak{m}_k(\rho) .
		\]
		Using the same component comparison as in the proof of Lemma \ref{le4-2}, yields
        \begin{equation}\label{dSS1}
		\begin{aligned}
			S(s_k)-S(r_k)
			&\ge (\mathfrak{m}_k(\rho))^2-\sum_jm_{j,k}^2 \\
			&\ge
			\left(\sum_jm_{j,k}\right)^2-\sum_jm_{j,k}^2 
			=2\sum_{i<j}m_{i,k}m_{j,k} \ge2m_{1,k}m_{2,k}.
		\end{aligned}
         \end{equation}
		The right-continuous representative of $S$ satisfies
		\[
		S(s_k)\to S(t_*^-), \qquad S(r_k)\to S(t_*).
		\]
		Passing to the limit in \eqref{dSS1}, we obtain
		\[
		S(t_*^-)-S(t_*)\ge2m_1m_2.
		\]
		It follows from the definition of the Stieltjes measure that one has
		\[
		S(t_*^-)-S(t_*)=(-\dd S)(\{t_*\}),
		\]
		which proves \eqref{5-2}. The proof of the lemma is completed.
	\end{proof}
	
	\begin{proposition}\label{Pro5-2}
		If $\mu=0$, then $E_\Phi(t)$ is connected for every regular value $t<\M$.
	\end{proposition}
	
	\begin{proof}
		Suppose, to the contrary, that $E_\Phi(t)$ is disconnected for some regular value $t<\M$. Choose distinct components $D_t^1,D_t^2\in\cC_\Phi(t)$ and points $x_i\in D_t^i$, $i=1,2$. Let $\Gamma$ be a compact rectifiable
		curve joining $x_1$ to $x_2$. Since regular values are dense, we may choose a regular value
		$a<\min_{x\in\Gamma}\Phi(x)$.
		Then $\Gamma\subset E_a$, and hence
		$D_a(x_1)=D_a(x_2)$.
		Define 
		\begin{equation}\label{5-3}
			t_*:=\sup\{s\in[a,t]:D_s(x_1)=D_s(x_2)\}.
		\end{equation}
		The set on the right is nonempty. Moreover, $t_*>a$, because $\Gamma\subset E_\Phi(s)$ for all $s>a$ sufficiently close to
		$a$. Since $t$ is regular and $x_1,x_2$ lie in distinct components of $E_\Phi(t)$, Lemma \ref{le3-2} shows that they remain separated on a neighborhood of $t$; thus $t_*<t$.
		
		The value $t_*$ must be critical. Indeed, if it were regular, the component correspondence of Lemma \ref{le3-2} would preserve, across
		a neighborhood of $t_*$, whether $x_1$ and $x_2$ belong to the same component. This contradicts the definition of $t_*$.
		
		Choose regular values $\left\{s_k\right\}$ and $\left\{r_k\right\}$ satisfying \eqref{rbj} for $t_*$ defined in \eqref{5-3} and set
		\begin{equation}\label{5-5}
			M_k:=D_{s_k}(x_1)=D_{s_k}(x_2), \qquad D_{i,k}:=D_{r_k}(x_i), \quad i=1,2.
		\end{equation}
		Then $D_{1,k}$ and $D_{2,k}$ are distinct components contained in
		$M_k$, and the nesting of superlevel components gives
		\[
		M_{k+1}\subset M_k, \qquad D_{i,k}\subset D_{i,k+1}, \quad i=1,2.
		\]
		Since $r_k<t$, we also have $D_t^i\subset D_{i,k}$. Therefore,
		it follows from \eqref{5-1} that
		\begin{equation}\label{5-6}
			\inf_k\int_{D_{i,k}}\rho(x)\,\dd x \ge \int_{D_t^i}\rho(x)\,\dd x>0, \qquad i=1,2.
		\end{equation}
		Using Lemma \ref{le5-1}, yields
			$\mu(\{t_*\})=(-\dd S)(\{t_*\})>0$.
		This contradicts $\mu=0$. The proof of the proposition is thus completed.
	\end{proof}
	
	Combining Proposition \ref{Pro5-2} with Lemma \ref{le3-1}, we conclude that, if $\mu=0$, then $\Phi^{-1}(t)$ is a single $C^2$ Jordan curve for every regular value
	$t<\M$.
	
	\subsection{A semilinear relation}
	We now show that, when $\mu=0$, there exists a global functional relation between $\rho$ and $\Phi$: namely, $\rho=f(\Phi)$ for some Borel function $f$ defined on the range of $\Phi$. Consequently,
	\eqref{3-1} reduces to a standard semilinear elliptic equation.
	
	\begin{proposition}\label{Pro5-3}
		If $\mu=0$, then there exists a bounded Borel function
		$f:(-\infty,\M]\to[0,\infty)$
		such that
		\begin{equation}\label{5-7}
			\rho(x)=f(\Phi(x)) \qquad\text{for a.e. }x\in\R^2.
		\end{equation}
	\end{proposition}
	
	\begin{proof}
		For every regular value $t$, $\Phi^{-1}(t)$ is a single $C^2$ Jordan curve. It follows from \eqref{3-1} that
		the function $\rho$ is constant on $\Phi^{-1}(t)$. Let us denote this value by $c(t)$. Let $N:=(-\infty,\M]\setminus\mathscr R_\Phi$ be the set of critical values, and define
		\begin{equation}\label{5-8}
			f(t):=
			\begin{cases}
				c(t),&t\in\mathscr R_\Phi,\\
				0,&t\in N.
			\end{cases}
		\end{equation}
		On each connected component of $\mathscr R_\Phi$, $c$ is continuous. Hence $f$ is Borel measurable, and satisfies
		\[
		0\le f\le\|\rho\|_{L^\infty}.
		\]
		By construction, \eqref{5-7} holds whenever $\Phi(x)$ is a regular value. On $\Phi^{-1}(N)$, we have $\rho=0$ almost everywhere, while $f(\Phi)=0$ by \eqref{5-8}.
		Therefore \eqref{5-7} holds almost everywhere in $\R^2$. This finishes the proof of the proposition.
	\end{proof}

	\subsection{Radial symmetry}
	
	\begin{lemma}\label{le5-5}
		If $\mu=0$, then for almost every regular value $t$ with $f(t)>0$, the set $E_\Phi(t)$ is a disk and $|\nabla\Phi|$ is constant on $\partial E_\Phi(t)$.
	\end{lemma}
	
	\begin{proof}
		By Proposition \ref{Pro5-2}, $E_\Phi(t)$ is connected for every regular value $t$. Recall that
		\[ a(t)=|E_\Phi(t)|, \qquad m(t)=\int_{E_\Phi(t)}\rho, \qquad c(t)=f(t). \]
		Hence one has
		\begin{equation}\label{5-9}
			S(t)=m^2(t)
			\quad\text{and}\quad
			A(t)=f(t)a(t)
		\end{equation}
		at every regular value. Since $\mu=0$, the decomposition \eqref{4-13} implies that, for almost every regular $t$,
		\begin{equation}\label{5-10}
			-\frac{\dd}{\dd t}m^2(t)=8\pi f(t)a(t).
		\end{equation}
		Thus equality holds in \eqref{4-4}. Whenever $f(t)>0$, its equality characterization shows that $E_\Phi(t)$ is a disk and $|\nabla\Phi|$ is constant on $\partial E_\Phi(t)$. The proof of the lemma is completed.
	\end{proof}
	
	Note that the nonlinearity $f$ obtained in Proposition~\ref{Pro5-3} is merely bounded, nonnegative, and Borel measurable. We therefore need a symmetry theorem that remains valid when the nonlinearity has such low regularity. The corresponding theory originates in the pioneering work of Lions \cite{Lions1981}, who established the two-dimensional case and conjectured its validity in arbitrary dimensions. Serra \cite{Serra2013} subsequently proved higher-dimensional versions under an additional structural assumption on the nonlinearity, and the conjecture was recently resolved by Zhan \cite{Zhan2026}. Quasilinear extensions can be found in \cite{Kesavan1994,Serra2013,Zhan2026}.
	
	For completeness, we record below the precise planar result needed here. It is essentially contained in \cite{Lions1981}, and also follows from the case $p=n=2$ of \cite[Theorem~1]{Serra2013}.

	\begin{proposition}\label{Li}
		Let $B\subset\R^2$ be a disk, and let $G\colon[0,\infty)\to[0,\infty)$ be a locally bounded Borel function. Suppose that $u\in C^1(B)\cap C(\overline B)$ is a nonnegative weak
		solution of
		$$
		\begin{cases}
			-\Delta u=G(u) & \text{in } B,\\
			u=0 & \text{on } \partial B.
		\end{cases}
		$$
		Then $u$ is radially symmetric about the center of $B$ and is nonincreasing as a function of the radial variable.
	\end{proposition}
	
	We are now in a position to establish the radial symmetry of both $\Phi$ and $\rho$ under the assumption that $\mu=0$.
	\begin{proposition}\label{Pro5-6}
		If $\mu=0$, then there exists a point $a\in\R^2$ such that both $\Phi$ and $\rho$ are radially symmetric with respect to $a$.
	\end{proposition}
	
	\begin{proof}
{\it Step 1.}		Set
		$$
		A:=\{t\in\mathscr R_\Phi:f(t)>0\}.
		$$
		We first note that $A\neq\varnothing$. Indeed, if $A=\varnothing$, then \eqref{5-7} and \eqref{5-8} would imply that $\rho=0$ almost everywhere in $\R^2$. Hence, by \eqref{3-1}, the function $\Phi$ would be harmonic in $\R^2$. Since $\Phi(x)\to-\infty$ as $|x|\to\infty$, it is bounded above, and therefore constant by Liouville's theorem. This leads to a contradiction. We may thus define $t_*:=\inf A$. Since $A$ is open and the conclusion of Lemma \ref{le5-5} holds for almost every $t\in A$, we may choose a sequence $t_k\in A$ with $t_k\searrow t_*$ such that, for each $k$, there exist $R_k>0$ and $a_k\in\R^2$ satisfying
$E_\Phi(t_k)=B_{R_k}(a_k)$.
		On $E_\Phi(t_k)$, set
		$
		u_k:=\Phi-t_k
		$
		and define $G_k\colon[0,\infty)\to[0,\infty)$ by
		$$
		G_k(s):=
		\begin{cases}
			f(t_k+s), & 0\leq s\leq \M-t_k,\\
			0, & s>\M-t_k.
		\end{cases}
		$$
		Since $f$ is locally bounded and Borel measurable, $G_k$ is a bounded
		nonnegative Borel function. For $x\in E_\Phi(t_k)$, one has
		$
		0<u_k(x)\leq \M-t_k.
		$
		Therefore, by Proposition \ref{Pro5-3},
		$$
		-\Delta u_k=f(\Phi)=f(t_k+u_k)=G_k(u_k) \qquad\text{a.e. in }E_\Phi(t_k).
		$$
		Thus $u_k$ is a weak solution of
		$$
		-\Delta u_k=G_k(u_k) \qquad\text{in }E_\Phi(t_k).
		$$
		Moreover, we have
		$
		u_k>0\text{ in }E_\Phi(t_k), \ u_k=0\text{ on }\partial E_\Phi(t_k).
		$
		Proposition \ref{Li} therefore implies that $u_k$, and hence $\Phi$, is radially symmetric with respect to $a_k$ and nonincreasing in the radial variable on $E_\Phi(t_k)$.
		
		{\it Step 2.} We next show that the centers $a_k$ are independent of $k$. If
		$t_{k+1}=t_k$, then
		$E_\Phi(t_{k+1})=E_\Phi(t_k)$,
		and the uniqueness of the center of a disk immediately gives $a_{k+1}=a_k$. Suppose now that $t_{k+1}<t_k$. Then
		$
		\varnothing\neq E_\Phi(t_k)\subsetneq E_\Phi(t_{k+1}).
		$
		Since $\Phi$ is radially symmetric and nonincreasing with respect to $a_{k+1}$ on $E_\Phi(t_{k+1})$, every nonempty proper strict superlevel set of $\Phi$ in $E_\Phi(t_{k+1})$ is a disk centered at $a_{k+1}$. In
		particular,
		$$
		E_\Phi(t_k)
		=
		\bigl\{x\in E_\Phi(t_{k+1}):\Phi(x)>t_k\bigr\}
		$$
		is centered at $a_{k+1}$. On the other hand,
		$E_\Phi(t_k)=B_{R_k}(a_k)$.
		The uniqueness of the center of a disk therefore yields $a_k=a_{k+1}$. Hence all the disks $E_\Phi(t_k)$ have a common center, which we denote by $a$.
		
		Since $t_k\searrow t_*$, we have
		$$
		\{\Phi>t_*\}=\bigcup_{k\geq1}E_\Phi(t_k).
		$$
		The sets on the right-hand side are nested disks with common center $a$. Consequently, $\{\Phi>t_*\}$ is either a disk centered at $a$ or the whole plane. Moreover, $\Phi$ is radially symmetric with respect to $a$ and nonincreasing in the radial variable on this set.
		
{\it Step 3.}		Since $\rho=f(\Phi)$ almost everywhere, it follows that $\rho$ is radially symmetric with respect to $a$ on $\{\Phi>t_*\}$. We also have
		$$
		\rho=0 \qquad\text{a.e. on }\R^2\setminus\{\Phi>t_*\}.
		$$
		Thus $\rho$ is radially symmetric with respect to $a$ in $\R^2$. It remains to prove the same conclusion for $\Phi$. We already know that $\Phi$ is radially symmetric on $\{\Phi>t_*\}$. If this set is the whole plane, there is nothing more to prove. Otherwise, it is a disk $B_R(a)$. Since $\rho$ vanishes almost everywhere outside $B_R(a)$, it is compactly supported. Let
		$
		\mathcal N(x):=\frac{1}{2\pi}\log\frac{1}{|x|}
		$
		be the two-dimensional Newtonian potential, and set
		$
		h:=\Phi-\mathcal N*\rho.
		$
		Then $h$ is harmonic in $\R^2$. Since both $\Phi$ and $\mathcal N*\rho$ are radially symmetric with respect to $a$ in $B_R(a)$, so is $h$. A radially symmetric harmonic function in $\R^2$ must be constant. Hence $h$ is constant in $B_R(a)$ and, by unique continuation for harmonic functions, in all of $\R^2$. Consequently,
		$
		\Phi=\mathcal N*\rho+C
		$
		for some $C\in\R$. It follows that $\Phi$ is radially symmetric with respect to $a$ in $\R^2$. Hence the proof of the proposition is completed.
	\end{proof}

	\section{Proof of Theorem \ref{main} with $\Omega=0$}\label{s6}
	In this section, we prove Theorem \ref{main} in the case $\Omega=0$. The key idea is to apply the rigidity theory developed in the preceding three sections.
	
	By \eqref{threshold}, $\omega_0$ has a fixed sign. Choose
	$\sigma\in\{-1,1\}$ and set
	\begin{equation}\label{7-1}
		\omega:=\sigma\omega_0\geq0,
		\qquad
		v:=K*\omega,
		\qquad
		\Gamma:=\int_{\R^2}\omega\,\dd x.
	\end{equation}
	Then $v\cdot\nabla\omega=0$. If $\Gamma=0$, then $\omega\equiv0$, and the conclusion is immediate. We henceforth assume that $\Gamma>0$.
	
	\subsection{The stationary stream function}
	We first construct a stationary stream function satisfying the asymptotic behavior at infinity in \eqref{3-1}.
	\begin{lemma}\label{le7-1}
		There exists a stream function $u\in C_{\loc}^{2,\beta}(\R^2)$ for every $0<\beta<1$, unique up to an additive constant, such that
		\begin{equation}\label{7-2}
			v=\nabla^\perp u,\qquad-\Delta u=\omega.
		\end{equation}
		Under the normalization $u(0)=0$,
		\begin{equation}\label{7-3}
			u(x)=-\frac{1}{2\pi}\int_{\R^2}\log\frac{|x-y|}{|y|}\,\omega(y)\,\dd y.
		\end{equation}
		Moreover,
		\begin{equation}\label{7-4}
			u(x)\to-\infty\qquad\text{as }|x|\to\infty.
		\end{equation}
		Consequently, every superlevel set of $u$ is bounded, and $u$ attains a finite global maximum.
	\end{lemma}
	
	\begin{proof}
		Define $u$ by \eqref{7-3}. The logarithmic singularities in the integrand are locally integrable. Moreover, for each fixed $x\in\R^2$
		and $|y|>2|x|+1$, we have
		$
		\left|\log\frac{|x-y|}{|y|}\right| \leq C\frac{|x|}{|y|}.
		$
		Together with $\omega\in L^1(\R^2)\cap L^\infty(\R^2)$, these observations show that the integrand in \eqref{7-3} belongs to $L^1(\R^2)$ as a function of $y$. Hence $u$ is well defined.
		
		Standard potential theory and local elliptic regularity give \eqref{7-2} and the claimed regularity. The identity $v=\nabla^\perp u$ also yields uniqueness up to an additive constant.
		
		To prove \eqref{7-4}, choose $R>1$ such that
		$$
		m:=\int_{B_R}\omega\,\dd x>\frac{\Gamma}{2},
		\qquad \eta:=\int_{\R^2\setminus B_R}\omega\,\dd x=\Gamma-m<m.
		$$
		For $r:=|x|>2R$, the contribution of $B_R$ to \eqref{7-3} is bounded by
		$$
		-\frac{1}{2\pi}\int_{B}\log\frac{|x-y|}{|y|}\,\omega(y)\,\dd y\leq-\frac{m}{2\pi}\log r+C_R.
		$$
		Using
		$
		-\log\frac{|x-y|}{|y|} \leq \log^+\frac{|y|}{|x-y|}
		$
		and splitting the complement of $B$ into $R<|y|\leq2r$ and $|y|>2r$, yields
		$$
		-\frac{1}{2\pi}\int_{B^c}\log\frac{|x-y|}{|y|}\,\omega(y)\,\dd y\leq\frac{\eta}{2\pi}\log r+C_R.
		$$
		Hence
		$$
		u(x)\leq-\frac{m-\eta}{2\pi}\log r+C_R.
		$$
		Since $m-\eta=2m-\Gamma>0$, \eqref{7-4} follows. This the proof of the lemma.
	\end{proof}
	
	By \eqref{1-6} and \eqref{7-2}, we have
	\begin{equation}\label{7-5}
		\nabla^\perp u\cdot\nabla\omega=v\cdot\nabla\omega=0.
	\end{equation}
	Thus $(\Phi,\rho):=(u,\omega)$ satisfies \eqref{3-1}. In the rest of this section, $S$, $A$, $\mu$, and $\M$ denote the quantities
	associated with this pair and introduced in Sections \ref{s3}--\ref{s5}; in particular,
	$
	\M=\max_{\R^2}u.
	$
	
	\subsection{The normalized Bernoulli function}
	
	Let $H\in C^1(\R^2)$ be the Bernoulli function constructed in Section~\ref{s3}, so that
	\begin{equation}\label{add-26-9-12}
		\nabla H=\omega\nabla u.
	\end{equation}
	By Lemma \ref{le3-0}, we have $|\nabla u|\in L^\infty(\R^2)$, and hence
	$
	\nabla H\in L^1(\R^2)\cap L^\infty(\R^2).
	$
	By Lemma~\ref{lem7-7}, there exists a constant $c\in\R$ such that
	$$
	\|H-c\|_{L^2(\R^2)} \leq C\|\nabla H\|_{L^1(\R^2)}.
	$$
	Subtracting $c$ from $H$, which does not affect \eqref{add-26-9-12}, we may assume that $H\in L^2(\R^2)$. Moreover, since $\nabla H\in L^\infty(\R^2)$, the function $H$ is Lipschitz continuous. It follows that
	\begin{equation}\label{7-6}
		H(x)\to 0  \qquad\text{as }|x|\to\infty.
	\end{equation}

	Recalling Proposition \ref{Pro3-5}, we have a one-dimensional representation of $H$ along each branch. Using \eqref{7-6}, one can now remove the branchwise constants.
	
	For $t<0$, let $D_0(t)$ denote the connected component of $\{u>t\}$ containing the origin.
	\begin{lemma}\label{le7-3}
		For every $x\in\R^2$, we have
		\begin{equation}\label{7-8}
			H(x)=\int_{-\infty}^{u(x)}\mathfrak c(x,t)\,\dd t\geq0,
		\end{equation}
		where $\mathfrak c$ is the branchwise density defined by \eqref{3-8}. Consequently,
		\begin{equation}\label{7-9}
			\int_{\R^2}H(x)\,\dd x=\int_{-\infty}^{\M}A(t)\,\dd t,
		\end{equation}
		with both sides allowed to take the value $+\infty$.
	\end{lemma}
	
	\begin{proof}
		For a regular value $t<0$, set
		$
		h_0(t):=H|_{\partial D_0(t)}.
		$
		If $t<\min_{\overline{B_R}}u$, then $B_R\subset D_0(t)$. Hence $\partial D_0(t)$ escapes to infinity as $t\to-\infty$, and
		\eqref{7-6} gives $h_0(t)\to0$ along regular values. Fix $x\in\R^2$, and choose regular values $t_k\to-\infty$ below the minimum of $u$ on a compact path joining $0$ to $x$. Then
		$x\in D_0(t_k)$. By Proposition \ref{Pro3-5}, we have
		$$
		H(x)=h_0(t_k)+\int_{t_k}^{u(x)}\mathfrak c(x,t)\,\dd t.
		$$
		Letting $k\to\infty$ proves \eqref{7-8}. Note that \eqref{7-9} follows from Tonelli's theorem and \eqref{4-10}. The proof is thus complete.
	\end{proof}
	
	\subsection{Vanishing of the defect measure}
	In view of Proposition~\ref{Pro5-6}, to establish radial symmetry of $u$ and $\omega$, it suffices to prove that the defect measure $\mu$ vanishes.
	
	We begin by extracting from the mass-square decomposition an identity relating $\mu$ to the Bernoulli function.
	
	\begin{lemma}\label{lem7-8}
		The function $H$ belongs to $L^1(\R^2)$, and
		\begin{equation}\label{7-10}
			\Gamma^2= 8\pi\int_{\R^2}H\,\dd x +\mu((-\infty,\M)).
		\end{equation}
	\end{lemma}

	\begin{proof}
		First, we claim that
		\begin{equation}\label{7-11}
			\lim_{t\to-\infty}S(t)=\Gamma^2,\qquad \lim_{t\nearrow\M}S(t)=0.
		\end{equation}
		The second limit follows from Lemma~\ref{le4-2}. To prove the first,
		set
		$
		m_0(t):=\int_{D_0(t)}\omega\,\dd x.
		$
		For every fixed $R>0$, the inclusion $B_R\subset D_0(t)$ holds for all sufficiently negative $t$. Hence
		$
		\liminf_{t\to-\infty}m_0(t)  \geq \int_{B_R}\omega\,\dd x.
		$
		Letting $R\to\infty$ and using $\omega\in L^1(\R^2)$, we obtain
		$
		\liminf_{t\to-\infty}m_0(t)\geq\Gamma.
		$
		Since $m_0(t)\leq\Gamma$, it follows that $m_0(t)\to\Gamma$ as $t\to-\infty$. Finally, we have
		$$
		[m_0(t)]^2 \leq S(t) \leq\left(\int_{\{u>t\}}\omega\,\dd x\right)^2 \leq\Gamma^2,
		$$
		and the first limit in \eqref{7-11} follows. By \eqref{4-13}, we have the following identity of measures on $(-\infty,\M)$:
		$$
		-\dd S=8\pi A(t)\,\dd t+\mu.
		$$
		Integrating between regular endpoints, letting them tend respectively to $-\infty$ and $\M$, and using \eqref{7-9} and \eqref{7-11}, we obtain
		\begin{equation*}
			\Gamma^2=8\pi\int_{\R^2}H\,\dd x+\mu((-\infty,\M)).
		\end{equation*}
		Since $H\geq0$ and $\mu$ is a nonnegative measure, both terms on the right-hand side are finite. The proof of the lemma is thus completed.
	\end{proof}

In the following, we complete the proof of Theorem \ref{main} with $\Omega=0$.		By Lemma~\ref{lem7-8}, it remains to prove that
		$$
		\int_{\R^2}H\,\dd x=\frac{\Gamma^2}{8\pi}.
		$$
		We obtain this identity from a global Pohozaev argument. Set
		$$
		w(x):=u(x)+\frac{\Gamma}{2\pi}\log|x|.
		$$
		By \eqref{7-3}, we have
		$$
		\nabla w(x)
		=
		\int_{\R^2}
		\bigl[\kappa(x-y)-\kappa(x)\bigr]\omega(y)\,\dd y,\ \ \ \kappa(x):=-\frac{1}{2\pi}\frac{x}{|x|^2}.
		$$
		Lemma \ref{Lem2-3} therefore provides a sequence $R_k\to\infty$ such that
		\begin{equation}\label{7-13}
			\int_{\left\{R_k<|x|<2R_k\right\}}|\nabla w|^2\,\dd x\to 0.
		\end{equation}
		Since $H\geq0$ and $H\in L^1(\R^2)$, we also have
		$$
		\int_{\left\{R_k<|x|<2R_k\right\}}H\,\dd x\to0.
		$$
		It follows from intermediate value theorem that there exists
		$r_k\in(R_k,2R_k)$ such that
		\begin{equation}\label{7-14}
			r_k\int_{\partial B_{r_k}}
			\bigl(|\nabla w|^2+H\bigr)\,\dd\cH^1
			\to0.
		\end{equation}
		As in Subsection \ref{s3-3}, we introduce the following stress tensor
		$$
		\tilde{T}:= \nabla u\otimes\nabla u -\frac12|\nabla u|^2I +HI.
		$$
		Then one has
		$$
		\Div  \tilde{T}=0, \quad  \operatorname{tr} \tilde{T}=2H, \text{ and }\Div( \tilde{T}x)=2H.
		$$
		Integrating this identity over $B_r$ and applying the divergence
		theorem gives
		\begin{equation}\label{7-15}
			2\int_{B_r}H\,\dd x=\frac r2\int_{\partial B_r} \bigl(u_r^2-u_\tau^2+2H\bigr)\,\dd\cH^1.
		\end{equation}
		Note that, on $\partial B_r$, one has
		\begin{equation}\label{7-15-1}
			w_r=u_r+\frac{\Gamma}{2\pi r}, \qquad  u_\tau=w_\tau.
		\end{equation}
		It follows from the divergence theorem that
		\begin{equation}\label{7-16}
			\int_{\partial B_r}w_r\,\dd\cH^1= \int_{\partial B_r}u_r\,\dd\cH^1+\Gamma= \Gamma-\int_{B_r}\omega\,\dd x \to 0 \qquad\text{as }r\to\infty.
		\end{equation}
		Combining \eqref{7-15}, \eqref{7-15-1} and \eqref{7-16}, yields
		\begin{equation}\label{7-17}
			2\int_{B_r}H\,\dd x=\frac{\Gamma^2}{4\pi} +\frac r2\int_{\partial B_r}\bigl(w_r^2-w_\tau^2+2H\bigr)\,\dd\cH^1-\frac{\Gamma}{2\pi}\int_{\partial B_r}w_r\,\dd\cH^1
		\end{equation}
		Taking $r=r_k$ and letting $k\to\infty$, we infer from \eqref{7-14} and \eqref{7-17} that
		$$
		\int_{\R^2}H\,\dd x=\frac{\Gamma^2}{8\pi}.
		$$
		The proof is thus complete.
	\qed
	
	We have therefore completed the proof of Theorem \ref{main} in the case $\Omega=0$.
	\begin{remark}
	    From the proof of the steady state case above, it can be seen that our method also provides an alternative proof of the classical result in \cite{ChenLi1991} that does not employ the method of moving planes. In fact, it follows from the Lemma~1.2 in \cite{ChenLi1991} that, in this case, $u$ is bounded above and satisfies the general elliptic problem considered in Section~\ref{s3}. The corresponding Bernoulli function $H$ is then simply $e^u$.
	\end{remark}
	\section{Proof of Theorem \ref{main} with $\Omega\not=0$}\label{s7}
	In this section, we prove Theorem \ref{main} in the case $\Omega\neq0$. The main point is to introduce a normalized stream function so that the rigidity theory of Sections \ref{s3}--\ref{s5} applies globally.

	\subsection{A normalized stream function}
	We first introduce a normalized stream function. Recall that, in Theorem \ref{main}, $\omega_0$ is assumed only to belong to $C^1(\mathbb{R}^2)\cap L^1(\R^2)\cap L^\infty(\R^2)$. Under this assumption, the convolution $\mathcal{N}*\omega_0$ may not be well defined; see, for instance,
	\cite{Lieb2001}. To normalize the behavior at infinity, define
	\begin{equation}\label{2-6}
		\psi(x)=\frac1{2\pi}\int_{\R^2} \log\frac{\langle y\rangle}{|x-y|}\,\omega_0(y)\,\dd y,
	\end{equation}
    where $\langle y\rangle := (1+|y|^2)^{1/2}$.
	The integral in \eqref{2-6} is absolutely convergent for every $x\in\R^2$. Moreover, the straightforward computations show that
	\begin{equation}\label{2-6-1}
		\psi(x)=o(|x|),\ \ \  \nabla\psi(x)=-\frac{1}{2\pi}\int_{\R^2} \frac{x-y}{|x-y|^2}\,\omega_0(y)\,\dd y=o(1) \qquad\text{as } |x|\to\infty.
	\end{equation}
	We also have
	\[
	-\Delta\psi=\omega_0\qquad\text{in }\R^2.
	\]
	Since $\omega_0\in C^1(\R^2)$, standard interior elliptic regularity yields $\psi\in C_{\loc}^{2,\alpha}(\R^2)$ for every $0<\alpha<1$.
	
	Set
	\begin{equation*}
		\Psi(x)=\psi(x)+\frac\Omega2|x|^2.
	\end{equation*}
	Then we have $ \nabla^\perp\Psi\cdot\nabla\rho=0$ by \eqref{1-6}. We now recast the problem in the form of \eqref{3-1}.
	
	\begin{lemma}\label{le6-1}
		Assume, in addition to the hypotheses of Theorem \ref{main}, that
		$\Omega\neq0$. Set
		$$
		\sigma:=-\operatorname{sgn}\Omega, \quad b:=2|\Omega|, \quad g:=\sigma\omega_0, \text{ and }\ \varphi:=\sigma\psi.
		$$
		Then $b>0$, $\sigma\in\{-1,1\}$, and
		\begin{equation}\label{6-2}
			\Phi:=\sigma\Psi =\varphi-\frac{b}{4}|x|^2, \qquad \rho:=-\Delta\Phi=b+g\geq0.
		\end{equation}
		$g\in C^1(\R^2)\cap L^1(\R^2)\cap L^\infty(\R^2)$. Moreover, the following conditions hold.
		\begin{equation}\label{6-3}
			\begin{gathered}
				\nabla\varphi(x)\to0,
				\ \ \
				\varphi(x)=o(|x|),\ \ \  \Phi(x)\to -\infty \quad\text{as }|x|\to\infty,\\
				\nabla^\perp\Phi\cdot\nabla\rho=0 \quad\text{in }\R^2.
			\end{gathered}
		\end{equation}
	\end{lemma}
	
	\begin{proof}
		Since $\omega_0\in L^\infty(\R^2)\cap L^1(\R^2)\cap C^1(\R^2)$, it follows that $\operatorname*{inf}_{\R^2}\omega_0 \leq0 \leq \operatorname*{sup}_{\R^2}\omega_0$. It follows from
		\eqref{threshold} and $\Omega\neq0$ that either
		$$
		\Omega<0
		\quad\text{and}\quad
		\omega_0-2\Omega\geq0,
		$$
		or
		$$
		\Omega>0 \quad\text{and}\quad 2\Omega-\omega_0\geq0.
		$$
		With the above choice of $\sigma$ and $b$, we have $b=-2\sigma\Omega>0$. Consequently,
		$$
		\Phi=\sigma\Psi =\sigma\psi+\frac{\sigma\Omega}{2}|x|^2 =\varphi-\frac{b}{4}|x|^2,
		$$
		and
		$$
		\rho=-\Delta\Phi =\sigma\omega_0-2\sigma\Omega=g+b\geq0.
		$$
		The asymptotic properties in \eqref{6-3} follow from \eqref{2-6-1}. Finally, since
		$\Phi=\sigma\Psi$ and $\rho=b+\sigma\omega_0$, it follows that
		$$
		\nabla^\perp\Phi\cdot\nabla\rho=\nabla^\perp(\sigma\Psi)\cdot\nabla(b+\sigma\omega_0)=\nabla^\perp\Psi\cdot\nabla\omega_0=0.
		$$
		This completes the proof.
	\end{proof}
	
	In view of Proposition \ref{Pro5-6}, it remains only to show that the defect measure $\mu$ introduced in Proposition \ref{Pro4-4} vanishes. Indeed, once $\mu=0$ is established, Proposition \ref{Pro5-6} implies that $\Phi$ and $\rho$ are radially symmetric with respect to some point $a\in\R^2$. The representation \eqref{6-2}, together with the asymptotic properties in \eqref{6-3}, then forces $a=0$.

	Our remaining task is therefore to prove that $\mu=0$. To this end, we first establish several preliminary lemmas.
	
	Since the quadratic term dominates $\Phi$ at infinity, its level sets are asymptotically circular. More precisely, we have the following lemma.
	
	\begin{lemma}\label{le6-2}
		There exists $R_0>0$ such that
		\begin{equation}\label{6-10}
			\partial_r\Phi(re^{i\theta}) \leq-\frac{b}{4}r \qquad \text{for all }r\geq R_0\text{ and }\theta\in[0,2\pi].
		\end{equation}
		Consequently, there exists a $t_0\in\R$ such that, for every $t<t_0$, the
		set
		$
		E_\Phi(t):=\{x\in\R^2:\Phi(x)>t\}
		$
		is a bounded domain, star-shaped with respect to the origin, and its
		boundary is a $C^2$ radial graph of the form
		\begin{equation}\label{6-11}
			\partial E_\Phi(t)=\bigl\{r_t(\theta)e^{i\theta}:\theta\in\T\bigr\},
		\end{equation}
		where $r_t\in C^2(\T)$ and $r_t>R_0$. In particular, every $t<t_0$ is a regular value of $\Phi$.
	\end{lemma}
	
	\begin{proof}
		Since $\nabla\varphi(x)\to0$ as $|x|\to\infty$, after enlarging $R_0$
		if necessary, we may assume that
		$$
		|\nabla\varphi(x)| \leq\frac{b}{4}|x| \qquad \text{for }|x|\geq R_0.
		$$
		Thus, for $r\geq R_0$,
		$$
		\partial_r\Phi(re^{i\theta})=\partial_r\varphi(re^{i\theta})-\frac{b}{2}r
		\leq |\nabla\varphi(re^{i\theta})|-\frac{b}{2}r\leq-\frac{b}{4}r.
		$$
		This proves \eqref{6-10}. Integrating along each ray, we obtain
		$$
		\Phi(re^{i\theta})\leq\max_{|x|=R_0}\Phi(x)-\frac{b}{8}\bigl(r^2-R_0^2\bigr),\qquad r\geq R_0.
		$$
		Hence $\Phi(re^{i\theta})\to-\infty$ uniformly in $\theta$.
		
		Choose
		$
		t_0<\min_{\overline{B_{R_0}}}\Phi.
		$
		Fix $t<t_0$ and $\theta\in\T$. The function $r\mapsto\Phi(re^{i\theta})$ is strictly decreasing on $[R_0,\infty)$, takes a value larger than
		$t$ at $r=R_0$, and tends to $-\infty$ as $r\to\infty$. Hence there is
		 a unique $r_t(\theta)>R_0$ such that
		$
		\Phi\bigl(r_t(\theta)e^{i\theta}\bigr)=t.
		$
		Since $B_{R_0}\subset E_\Phi(t)$, it follows that
		$$
		E_\Phi(t)=\bigl\{re^{i\theta}:\theta\in\T,\ 0\leq r<r_t(\theta)\bigr\}.
		$$
		Thus $E_\Phi(t)$ is bounded and star-shaped with respect to the origin. Moreover, it follows from \eqref{6-10} that the level crossing is transverse. The implicit function theorem therefore gives $r_t\in C^2(\T)$.
		Since $\partial_r\Phi\neq0$ on $\partial E_\Phi(t)$, the value $t$ is regular. The proof is thus complete.
	\end{proof}

	We next quantify the variation of $\Phi$ across a large annulus. For
	$R\geq R_0$, define
	\begin{equation}\label{6-12}
		t_R^-:=\max_{\theta\in\T}\Phi(2Re^{i\theta}),
		\qquad
		t_R^+:=\min_{\theta\in\T}\Phi(Re^{i\theta}),
		\qquad
		J_R:=(t_R^-,t_R^+).
	\end{equation}
	The next lemma shows that the length of $J_R$ is bounded below by a constant multiple of $R^2$ and that every level curve corresponding to $t\in J_R$ is confined to the annulus $R<|x|<2R$.
	
	\begin{lemma}\label{le6-3}
		For all sufficiently large $R$, the interval $J_R$ is nonempty, and satisfies
		\begin{equation}\label{6-13}
			|J_R|\geq c_bR^2,
		\end{equation}
		where $c_b>0$ depends only on $b$, and is contained in $(-\infty,t_0)$. Moreover, for every $t\in J_R$,
		\begin{equation}\label{6-14}
			R<r_t(\theta)<2R \qquad \text{for all }\theta\in\T.
		\end{equation}
	\end{lemma}
	
	\begin{proof}
		Since $\varphi(x)=o(|x|)$ as $|x|\to\infty$, one has
		\begin{align*}
			t_R^+-t_R^-
			&\geq
			\left(-\frac{b}{4}R^2-\sup_{|x|=R}|\varphi(x)|\right)-\left(-bR^2+\sup_{|x|=2R}|\varphi(x)|\right)
			=\frac{3b}{4}R^2-o(R).
		\end{align*}
		Consequently, for all sufficiently large $R$, we have
		$$
		|J_R|=t_R^+-t_R^-\geq\frac{b}{2}R^2.
		$$
		Thus \eqref{6-13} holds with $c_b=b/2$, and $J_R$ is nonempty. Furthermore, the uniform decay of $\Phi$ gives $t_R^+\to-\infty$ as $R\to\infty$, so $J_R\subset(-\infty,t_0)$ for all sufficiently large $R$.
		
		Finally, if $t\in J_R$, then, for every $\theta\in\T$, we have 
		$$
		\Phi(Re^{i\theta})\geq t_R^+>t,\qquad \Phi(2Re^{i\theta})\leq t_R^-<t.
		$$
		By \eqref{6-10}, the unique solution of $\Phi\bigl(r_t(\theta)e^{i\theta}\bigr)=t$ must satisfy
		$
		R<r_t(\theta)<2R.
		$
		This proves \eqref{6-14}. The proof of the lemma is completed.
	\end{proof}

	\subsection{Vanishing of the defect measure}
	
	For $t<t_0$, Lemma \ref{le6-2} shows that $E_\Phi(t)$ is a bounded star-shaped
	domain with boundary
	$$
	\Gamma_t:=\partial E_\Phi(t) =\bigl\{r_t(\theta)e^{i\theta}:\theta\in\T\bigr\}.
	$$
	Set
	\begin{equation}\label{6-15}
		L_t(\theta):=\bigl(r_t(\theta)^2+r_t'(\theta)^2\bigr)^{1/2}\text{ and } G_t(\theta) :=\left|\nabla\Phi\bigl(r_t(\theta)e^{i\theta}\bigr)\right|.
	\end{equation}
	The outward unit normal and the arclength element satisfy
	\begin{equation}\label{6-16}
		n_t=\frac{r_te_r-r_t'e_\theta}{L_t},\qquad\dd\cH^1=L_t\dd\theta,
		\qquad (x\cdot n_t)\dd\cH^1=r_t^2\dd\theta.
	\end{equation}
	
	To relate the defect $\delta(t)$ to the asymptotic circularity of the low-level sets, we estimate it in terms of the angular variation of $\varphi$ along $\partial E_\Phi(t)$. The following lemma provides the key quantitative bound.
	\begin{lemma}\label{le6-4}
		There exists $t_1<t_0$ such that, for every $t<t_1$, we have
		\begin{equation}\label{6-17}
			0\leq\delta(t)\leq4\pi\int_0^{2\pi}\left|\varphi_\theta\bigl(r_t(\theta),\theta\bigr)\right|^2\dd\theta.
		\end{equation}
	\end{lemma}
	
	\begin{proof}
		Since $E_\Phi(t)$ is connected, one has
		$$
		S(t)=\left(\int_{E_\Phi(t)}\rho\right)^2.
		$$
		Using the divergence theorem, \eqref{4-18}, and \eqref{6-16}, we get
		\begin{equation}\label{6-18}
			s(t)=\left(\int_{\Gamma_t}|\nabla\Phi|\,\dd\cH^1\right)^2=\left(\int_0^{2\pi}G_tL_t\,\dd\theta\right)^2,\qquad8\pi I(t)=2\pi\int_0^{2\pi}r_t^2G_t^2\,\dd\theta.
		\end{equation}
		Hence, by the Cauchy--Schwarz inequality, we have
		\begin{equation}\label{6-19}
			\delta(t)=S(t)-8\pi I(t) \notag \leq2\pi\int_0^{2\pi}G_t^2L_t^2\,\dd\theta-2\pi\int_0^{2\pi}r_t^2G_t^2\,\dd\theta \notag=2\pi\int_0^{2\pi}G_t^2r_t'^2\,\dd\theta.
		\end{equation}
		Differentiating $\Phi\bigl(r_t(\theta),\theta\bigr)=t$ with respect to $\theta$ yields $\Phi_r r_t'+\Phi_\theta=0$. Since the quadratic part of $\Phi$ is radial, $\Phi_\theta=\varphi_\theta$. Moreover, we have
		$$
		G_t^2=\Phi_r^2+\frac{\Phi_\theta^2}{r_t^2}=\Phi_r^2\frac{L_t^2}{r_t^2}.
		$$
		Consequently,
		\begin{equation}\label{6-20}
			G_t^2r_t'^2=|\varphi_\theta|^2\frac{L_t^2}{r_t^2},
		\end{equation}
		where the quantities on the right-hand side are evaluated at
		$\bigl(r_t(\theta),\theta\bigr)$.
		
		As $t\to-\infty$, the curves $\Gamma_t$ escape uniformly to infinity.
		Using \eqref{6-10}, one has
		$$
		\frac{|r_t'|}{r_t}=\frac{|\varphi_\theta|}{r_t|\Phi_r|}\leq\frac{|\nabla\varphi|}{|\Phi_r|}\leq\frac{4|\nabla\varphi|}{br_t}\to0\quad\text{uniformly in $\theta$.}
		$$
		 Thus, after decreasing $t_1$ if necessary, we have
		$$
		\frac{L_t^2}{r_t^2}=1+\frac{r_t'^2}{r_t^2}\leq2\qquad\text{for every }t<t_1.
		$$
		Combining this with \eqref{6-19} yields \eqref{6-17}. The proof is thus complete.
	\end{proof}

	We are now ready to complete the vanishing argument.
	
	\begin{proposition}\label{Pro6-5}
		For every regular value $t<\M$, $\delta(t)=0$. Consequently, the defect measure $\mu$ in Proposition \ref{Pro4-4} vanishes identically.
	\end{proposition}
	
	\begin{proof}
		In view of Proposition \ref{Pro4-4}, we have $\delta(t):=\mu((t,\M))$. Set
		\begin{equation*}
			w(x):=\psi(x)-\frac{1}{2\pi} \left(\int_{\R^2}\omega_0(y)\,\dd y\right) \log\frac{1}{|x|}.
		\end{equation*}
		Then
		\begin{equation*}
			\nabla w(x)=\int_{\R^2}\bigl[\kappa(x-y)-\kappa(x)\bigr]\omega_0(y)\,\dd y.
		\end{equation*}
		where $\kappa(x)$ is defined by $\kappa(x):=-\frac{1}{2\pi}\frac{x}{|x|^2}$. Let $R_k\to\infty$ be the sequence furnished by Lemma \ref{Lem2-3} with $Q=\nabla w$ and $g=\omega_0$. Since $|w_\theta|=|\varphi_\theta|$, we have 
		\begin{equation}\label{6-22}
			\mathcal E_k:=\int_{R_k<|x|<2R_k}|\nabla w|^2\,\dd x\to0,\ \ \ \int_{R_k}^{2R_k}\int_0^{2\pi}\frac{|\varphi_\theta(r,\theta)|^2}{r}\,\dd\theta\dd r \leq\mathcal E_k.
		\end{equation}
		Write $J_k:=J_{R_k}$. By Lemma \ref{le6-2}, for all sufficiently large
		$k$, we have
		$$
		J_k\subset(-\infty,t_1),\qquad|J_k|\geq c_bR_k^2,
		$$
		and
		$$
		R_k<r_t(\theta)<2R_k\qquad\text{for }(t,\theta)\in J_k\times\T.
		$$
		In particular, every $t\in J_k$ is a regular value. For each fixed $\theta$, the map
		$
		r\mapsto\Phi(re^{i\theta})
		$
		is strictly decreasing on $[R_k,2R_k]$. Its inverse on $J_k$ is $t\mapsto r_t(\theta)$, so the change of variables
		$$
		t=\Phi(r,\theta),\qquad\dd t=\Phi_r(r,\theta)\dd r
		$$
		is valid. Let $I_{k,\theta}\subset(R_k,2R_k)$ denote the interval corresponding to $J_k$. Integrating \eqref{6-17} over $J_k$ implies
		\begin{align}
			\int_{J_k}\delta(t)\,\dd t&\leq4\pi\int_0^{2\pi}\int_{I_{k,\theta}}|\varphi_\theta(r,\theta)|^2\bigl(-\Phi_r(r,\theta)\bigr)\,\dd r\dd\theta \notag\\
			&\leq CR_k\int_{R_k}^{2R_k}\int_0^{2\pi}|\varphi_\theta(r,\theta)|^2\,\dd\theta\dd r \notag\\
			&\leq CR_k^2\int_{R_k}^{2R_k}\int_0^{2\pi}\frac{|\varphi_\theta(r,\theta)|^2}{r}\,\dd\theta\dd r \notag\\
			&\leq CR_k^2\mathcal E_k.
			\label{6-23}
		\end{align}
		Here we used
		$$
		-\Phi_r=\frac{b}{2}r-\varphi_r\leq CR_k\qquad\text{on }R_k<r<2R_k
		$$
		for all sufficiently large $k$. Since $|J_k|\geq c_bR_k^2$, there exists a $t_k\in J_k$ such that
		\begin{equation}\label{6-24}
			0\leq\delta(t_k)\leq\frac{1}{|J_k|}\int_{J_k}\delta(t)\,\dd t\leq C\mathcal E_k\to0.
		\end{equation}
		Moreover,
		$$
		t_k<t_{R_k}^+=-\frac{b}{4}R_k^2+o(R_k),
		$$
		and hence $t_k\to-\infty$. Fix any regular value $t<\M$. For all sufficiently large $k$, we have $t_k<t$. Since $\delta$ is nonincreasing, we have
		$$
		0\leq\delta(t)\leq\delta(t_k)\to0.
		$$
		This proves $\delta(t)=0$.
		
		Finally, choose a decreasing sequence of regular values
		$s_j\to-\infty$. Then
		$$
		\mu((s_j,\M))=\delta(s_j)=0\qquad\text{for every }j.
		$$
		Since the intervals $(s_j,\M)$ increase to $(-\infty,\M)$, continuity
		from below implies
		$$
		\mu((-\infty,\M))=\lim_{j\to\infty}\mu((s_j,\M))=0.
		$$
		Thus $\mu\equiv0$. This finishes the proof of the proposition.
	\end{proof}
	
	As we mentioned before, for $\Omega\neq 0$, Theorem \ref{main} follows from Proposition \ref{Pro5-6} and \ref{Pro6-5}.

	\appendix
	
	\section{Potential estimates} \label{app1}
	In this appendix, we collect several basic estimates used in the proof of the main theorem.
	
	\begin{lemma}\label{le3-0}
		For every $g\in L^1(\R^2)\cap L^\infty(\R^2)$, one has
		\begin{equation}\label{2-1}
			\sup_{x\in\R^2}\int_{\R^2}\frac{|g(y)|}{|x-y|}\dd y\le (2\pi+1)\|g\|_{L^1}^{1/2}\|g\|_{L^\infty}^{1/2}.
		\end{equation}
		Moreover, $K*g$ is continuous on $\R^2$, and
		\begin{equation}\label{2-2}
			\lim_{|x|\to\infty}K*g(x)=0 .
		\end{equation}
	\end{lemma}
	
	\begin{proof}
		Fix $x\in\R^2$ and $r>0$. We have
		\[
		\int_{\R^2}\frac{|g(y)|}{|x-y|}\dd y\le \|g\|_\infty\int_{|z|<r}\frac{\dd z}{|z|}  +\frac1r\|g\|_1\le 2\pi\|g\|_\infty r+\frac1r\|g\|_1.
		\]
		If $\|g\|_\infty=0$, then \eqref{2-1} is immediate. Otherwise, $\|g\|_1>0$, and \eqref{2-1} follows by taking $r=(\|g\|_1/\|g\|_\infty)^{1/2}$. The continuity of $K*g$ is standard. To prove \eqref{2-2}, fix $M>0$ and decompose
		\[
		g=g\1_{B_M}+g\1_{\R^2\setminus B_M}.
		\]
		If $|x|>2M$, then
		\[
		\bigl|K*(g\1_{B_M})(x)\bigr|\le \frac{C}{|x|}\|g\|_1.
		\]
		On the other hand, applying \eqref{2-1} to the second term yields
		\[
		\|K*(g\1_{\R^2\setminus B_M})\|_\infty\le C\|g\|_\infty^{1/2} \|g\1_{\R^2\setminus B_M}\|_1^{1/2}.
		\]
		Letting first $|x|\to\infty$ with $M$ fixed, and then $M\to\infty$, we obtain \eqref{2-2}. The proof is thus complete.
	\end{proof}
	
	Let
	\begin{equation}\label{2-0}
		\mathcal{N}(x)=\frac{1}{2\pi}\log\frac{1}{|x|},\ \ \ \kappa(x)=\nabla N(x)=-\frac{1}{2\pi}\frac{x}{|x|^2}.
	\end{equation}
	
	\begin{lemma}\label{lem3-2}
		If $E\subset\R^2$ is measurable with $0<|E|<\infty$, then for $g\in L^1(\R^2)\cap L^\infty(\R^2)$, one has
		\begin{equation}\label{2-3}
			\int_E |\kappa*g|^2\,\dd x\le C m^2
			\left[ 1+\log^+\left( \frac{|E|^{1/2}M^{1/2}}{m^{1/2}}
			\right)\right],
		\end{equation}
		where the right-hand side is understood to be zero when $m=0$, and
		$
		m:=\|g\|_{L^1},\  M:=\|g\|_{L^\infty}
		$.
	\end{lemma}
	
	\begin{proof}
		Since $|\kappa(x)|\le C|x|^{-1}$, the convolution is
		pointwise controlled by the Riesz potential of order one:
		\[
		|\kappa*g(x)|\le C\int_{\R^2}\frac{|g(y)|}{|x-y|}\,\dd y=C I_1(|g|)(x).
		\]
		The weak Hardy--Littlewood--Sobolev inequality(see, for instance, \cite{Stein1970,Lieb2001}). therefore implies
		\begin{equation}\label{2-4}
			\bigl|\{x\in\R^2:|\kappa*g(x)|>\lambda\}\bigr|\le\frac{C m^2}{\lambda^2},\qquad \lambda>0.
		\end{equation}

		On the other hand, the estimate \eqref{2-1} yields
		\begin{equation}\label{2-5}
			\|\kappa*g\|_{L^\infty}\le C M^{1/2}m^{1/2}=:L.
		\end{equation}
		There is nothing to prove if $m=0$, so assume $m>0$ and set
		$\lambda_0:=m|E|^{-1/2}$.
		Using the layer-cake representation(\cite{Lieb2001}) together with \eqref{2-4}, we obtain
		\begin{align*}
			\int_E |\kappa*g|^2\,\dd x
			&=
			\int_0^L2\lambda\,\bigl|E\cap\{|\kappa*g|>\lambda\}\bigr|\,\dd\lambda
			\le|E|\lambda_0^2+C m^2\int_{\lambda_0}^L\frac{\dd\lambda}{\lambda},
		\end{align*}
		where the last integral is understood to be zero if $L\le\lambda_0$.
		Since $|E|\lambda_0^2=m^2$, it follows that
		\begin{equation}\label{621}
		\int_E |\kappa*g|^2\,\dd x\le C m^2
		\left[1+\log^+\left(\frac{L}{\lambda_0}\right)
		\right].
		\end{equation}
		Finally, by \eqref{2-5}, one has
		\begin{equation}\label{622}
        \frac{L}{\lambda_0}\le C\frac{|E|^{1/2}M^{1/2}}{m^{1/2}}.
		\end{equation}
		Substituting \eqref{622} into \eqref{621} yields \eqref{2-3}, we obtain
		\[
		\int_E |\kappa*g|^2\,\dd x\le C m^2
		\left[1+\log^+\left(\frac{|E|^{1/2}M^{1/2}}{m^{1/2}}\right)\right].
		\] This finishes the proof of the lemma.
	\end{proof}
	
	For \(g\in L^1(\mathbb{R}^2)\cap L^\infty(\mathbb{R}^2)\), we introduce the renormalized potential
	\begin{equation}\label{2-9}
		Q(x)=\int_{\mathbb{R}^2}\bigl[\kappa(x-y)-\kappa(x)\bigr]g(y)\,\dd y .
	\end{equation}
	We will show that there exists a sequence of large dyadic annuli along which the $L^2$ energy of $Q$ vanishes. To this end, for each $j\geq2$, define
	\begin{equation}\label{2-11}
		\mathcal A_j:=\{x\in\R^2:2^j<|x|<2^{j+1}\}\text{ and } m_j:=\int_{\left\{2^{j-1}<|y|<2^{j+2}\right\}}|g(y)|\,\dd y.
	\end{equation}
	Since the enlarged annuli have uniformly bounded overlap, we have
	\begin{equation}\label{2-12}
		\sum_{j=2}^{\infty}m_j\leq4\|g\|_{L^1(\R^2)}.
	\end{equation}

	\begin{lemma}\label{Lem2-3}
		There exists a sequence of integers $j_k\to\infty$ such that
		\begin{equation}\label{2-13}
			\int_{\left\{2^{j_k}<|x|<2^{j_k+1}\right\}}|Q(x)|^2\,\dd x\to0\qquad\text{as }k\to\infty.
		\end{equation}
	\end{lemma}
	
	\begin{proof}
		Fix $j\ge2$ and set $R:=2^j$. Recalling \eqref{2-9}, we decompose $Q=I_j+L_j+O_j$, where
		\begin{align*}
			I_j(x)
			&:=\int_{\left\{|y|\leq R/2\right\}}\bigl[\kappa(x-y)-\kappa(x)\bigr]g(y)\,\dd y,\\
			L_j(x)
			&:=\int_{\left\{R/2<|y|<4R\right\}}\bigl[\kappa(x-y)-\kappa(x)\bigr]g(y)\,\dd y,\\
			O_j(x)
			&:=\int_{\left\{|y|\geq 4R\right\}}\bigl[\kappa(x-y)-\kappa(x)\bigr]g(y)\,\dd y.
		\end{align*}
		We estimate these three terms one by one.
        
		\smallskip
		\noindent\emph{Step 1. The inner region.}
		Let $x\in\mathcal A_j$ and $|y|<R/2$. For every $s\in[0,1]$, we have
		\[
		|x-sy|\ge |x|-s|y|\ge \frac R2.
		\]
		The mean value theorem, together with the estimate
		$|\nabla\kappa(z)|\le C|z|^{-2}$, therefore gives
		\[
		|\kappa(x-y)-\kappa(x)|\le C\frac{|y|}{R^2}.
		\]
		Consequently,
		\[
		\sup_{x\in\mathcal A_j}|I_j(x)|
		\le\frac{C}{R^2}\int_{|y|<R/2}|y|\,|g(y)|\,\dd y.
		\]
		Since $|\mathcal A_j|\le CR^2$, one has
		\begin{equation}\label{2-100}
			\int_{\mathcal A_j}|I_j|^2\,\dd x
			\le C\left[\frac1R\int_{|y|<R/2}|y|\,|g(y)|\,\dd y\right]^2
			\to0
			\qquad\text{as }j\to\infty,
		\end{equation}
		where the convergence follows from $g\in L^1(\R^2)$.

		\smallskip
		\noindent\emph{Step 2. The outer region.}
		If $x\in\mathcal A_j$ and $|y|>4R$, then
		\[
		|x-y|\ge |y|-|x|\ge\frac{|y|}{2}.
		\]
		Using $|\kappa(z)|\le C|z|^{-1}$ and $|x|\ge R$, yields
		\[
		|\kappa(x-y)-\kappa(x)|\le \frac{C}{R}.
		\]
		For $g\in L^1(\R^2)$, one has
		\begin{equation}\label{2-14}
			\int_{\mathcal A_j}|O_j|^2\,\dd x\le C\left(\int_{|y|>4R}|g(y)|\,\dd y\right)^2\to0\qquad\text{as }j\to\infty.
		\end{equation}

		\smallskip
		\noindent\emph{Step 3. The middle region.}
		Define
		\[
		h_j:=g\1_{\{R/2<|y|<4R\}},\qquad\gamma_j:=\int_{\R^2}h_j(y)\,\dd y.
		\]
		Then $L_j=\kappa*h_j-\gamma_j\kappa$. If $m_j=0$, then $h_j=0$ almost everywhere, and hence $L_j=0$.
		Suppose that $m_j>0$. Applying Lemma \ref{lem3-2} with
		$E=\mathcal A_j$, and recalling that
		$|\mathcal A_j|^{1/2}\le CR$, we find
		\begin{align*}
			\int_{\mathcal A_j}|\kappa*h_j|^2\,\dd x
			&\le Cm_j^2
			\left[
			1+\log^+\left(\frac{CR\|g\|_{L^\infty}^{1/2}}{m_j^{1/2}}
			\right)
			\right] \\
			&\le Cm_j^2\bigl(1+j+|\log m_j|\bigr),
		\end{align*}
		where the fixed quantity $\|g\|_{L^\infty}$ has been incorporated into the constant. Moreover, $|\gamma_j|\le m_j$, while
		\[
		\int_{\mathcal A_j}|\kappa(x)|^2\,\dd x\le C.
		\]
		Hence
		\[
		\int_{\mathcal A_j}|\gamma_j\kappa(x)|^2\,\dd x \le Cm_j^2.
		\]
		We therefore have
		\begin{equation}\label{2-15}
			\int_{\mathcal A_j}|L_j|^2\,\dd x\le Cm_j^2\bigl(1+j+|\log m_j|\bigr).
		\end{equation}
		
		Combining \eqref{2-100}, \eqref{2-14}, and
		\eqref{2-15}, we arrive at
		\begin{equation}\label{2-16}
			\int_{\mathcal A_j}|Q|^2\,\dd x
			\le\varepsilon_j+Cm_j^2\bigl(1+j+|\log m_j|\bigr),\qquad \varepsilon_j\to 0.
		\end{equation}
		By \eqref{2-12}, the sequence $(m_j)_{j\ge2}$ is summable. In particular,
		$m_j\to0$, and
		\[
		m_j^2\bigl(1+|\log m_j|\bigr)\to0,
		\]
		where, as usual, $r^2|\log r|$ is taken to be zero at $r=0$. It remains to handle the term $jm_j^2$. There must be a sequence
		$j_k\to\infty$ for which $j_km_{j_k}^2\to0$. Indeed, otherwise $\liminf_{j\to\infty}jm_j^2>0$, so that
		\[
		m_j\ge c j^{-1/2}.
		\]
		for all sufficiently large $j$ and some $c>0$. This is incompatible with \eqref{2-12}. Evaluating \eqref{2-16} at
		$j=j_k$, we obtain
		\[
		\int_{\mathcal A_{j_k}}|Q|^2\,\dd x \le \varepsilon_{j_k}+Cm_{j_k}^2+Cj_km_{j_k}^2+Cm_{j_k}^2|\log m_{j_k}|\to0
		\]
		as $k\to \infty$, which is \eqref{2-13}. The proof is thus complete.
	\end{proof}

	\begin{lemma}\label{lem7-7}
		Let $h\in W_{\loc}^{1,1}(\R^2)$ and suppose that
		$\nabla h\in L^1(\R^2;\R^2)$. Then there exists a unique constant
		$c_\infty\in\R$ such that
		$
		h-c_\infty\in L^2(\R^2)$ and
		$$\|h-c_\infty\|_{L^2(\R^2)}
		\leq
		C\|\nabla h\|_{L^1(\R^2)}.
		$$
	\end{lemma}
	
	\begin{proof}
		This is the case $n=2$ and $p=1$ of \cite[Theorem~1.78]{Maly1997}.
	\end{proof}
	
	\subsection*{Data Availability} No datasets were generated or analyzed during the current study.
	
	\subsection*{Competing Interests} The authors declare that they have no competing interests.

\end{document}